\documentclass[12pt]{amsart}

\AtBeginDvi{} %% しおりが文字化けしないように
\usepackage{amsmath,amssymb,amsfonts,amsthm,amscd}
\usepackage{mathtools}
\usepackage{mathrsfs}
\usepackage{array,bm}

\usepackage[dvipdfmx]{graphicx,xcolor}% Tikz ドライバ指定のた
\usepackage{tikz}
\usetikzlibrary{automata, positioning, arrows}

\newtheorem{thm}{Theorem}

\newtheorem{proposition}[thm]{Proposition}

\newtheorem{lem}{Lemma}
\theoremstyle{example}
\newtheorem{exam}{Example}
\theoremstyle{books}

\theoremstyle{definition}
\newtheorem{defn}{Definition}

\newtheorem*{conj}{Pisot Substitution Conjecture}

\newtheorem*{strongconj}{Strong Coincidence Conjecture}

\theoremstyle{remark}
\newtheorem{rem}{Remark}[section]

\usepackage{txfonts}
\usepackage{multimedia}

\def\vector#1{\mbox{\boldmath $#1$}}

\begin{document}

\title[Pisot substitution conjecture]
{Pisot Substitution Conjecture and Rauzy Fractals}

\author[K.~Nakaishi]{ Kentaro Nakaishi}

\address{Department of Mathematics, College of Humanities and Sciences,
Nihon University, Japan}
\email{nakaishi.kentarou@nihon-u.ac.jp}

\date{\today}


\begin{abstract}
We provide a proof for one version of Pisot conjecture.
We make use of the weak mixing property of the subshift of finite type derived from the prefix-suffix automaton 
to conclude that the substitution dynamical system has pure discrete spectrum.
\end{abstract}

\maketitle

%%%\tableofcontents

\section{Introduction}

Recurrence is one of the main concerns in Ergodic theory.
In this paper, we study the classification of
uniformly recurrent sequences generated by substitutions,
based on their spectral type.

There are a group of problems called {\it Pisot conjecture}
to which we recommend \cite{akiyama2015pisot} as a general reference.
Among various versions of Pisot conjecture and related problems,
we specifically address the followings.
\begin{conj}If $\sigma$ is an irreducible Pisot substitution, then
the substitution dynamical system has pure discrete spectrum (or
 pure point spectrum).
\end{conj}

\begin{strongconj}Every irreducible Pisot substitution $\sigma$ satisfies the strong coincidence
condition.
\end{strongconj}

We only mention known results which are directly related to these conjectures:
Barge and Diamond \cite{barge2002coincidence} proves Strong Coincidence Conjecture for
two symbol cases ($n=2$ in our notation) and, using it, Hollander and Solomyak \cite{hollander2003two}
solves Pisot Substitution Conjecture for $n=2$. 

Our goal is to prove the two conjectures for general $n\geq 2$
by a geometric and number-theoretical approach combined with Ergodic theory.
This `classical' approach initiated by Rauzy \cite{rauzy1982nombres}
consists of constructing a geometric representation of 
the substitution dynamical system (the Rauzy fractal $\mathcal{R}_{\sigma}$
and the domain exchange transformation
on it), and  showing that
it is  measure-theoretically
conjugate to an ergodic rotation of a compact abelian group.

Along the same lines,
there is another equivalent formulation which we favor in this paper. 
A collection of translations of subtiles $\mathcal{R}_{\sigma}(a)$ of the Rauzy fractal 
\[
\mathcal{T}=\{\mathcal{R}_{\sigma}(a)+\gamma: (\gamma, a)\in\Gamma\}
\]
forms a covering (multiple tiling) of the representation space $K_{\sigma}$.
It is known for irreducible Pisot substitutions
that the substitution dynamical system has pure discrete spectrum if and only if
$\mathcal{T}$ is a single covering or a {\it tiling}.

In previous research, 
general results concerning tiling are only proven {\it under additional conditions}.
The novelty of this paper is a {\it direct} proof of $\mathcal{T}$ being a tiling 
with help of the weak mixing property inherent in the subshift of finite type  
derived from the automaton.
As far as the author knows, this stronger property of ergodicity was overlooked and 
implications of it were not pursued in the literature.
For the sake of clarification, it should be noted that
the subshift does not commute with the substitution dynamical system, but rather
it does with the desubstitution \cite{mosse1996reconnaissabilite}.
It is this structure that determines whether tiling or not.

The paper is organized as follows: Section 2-7 lay the foundation for the conjectures.
The representation space $K_{\sigma}$ is built in the ad\`{e}le ring
of the number field $\mathbb{Q}(\alpha)$.
As Rauzy fractals carry a multitude of mathematical structures, so too can
several definitions of them be made.
The reason why we adopt a seemingly different definition of
Rauzy fractals from the conventional ones (Definition~\ref{rauz})
is that we focus on a structure which previous works neglected.
Their identity is confirmed in Remark~\ref{identity}, so that conventional results about
Rauzy fractals can be inherited.

In Section \ref{perturbation}, we start the proof of tiling.
When a multiple tiling situation
\[
\mathrm{int}\mathcal{R}_{\sigma}(a)\cap\mathrm{int}(\mathcal{R}_{\sigma}(b)+\gamma)\neq\emptyset
\]
occurs, 
the algebraic integer (Pisot number $\alpha$) enforces
restrictions on this intersection ($a=b$ and $\gamma=0$).
In other words, the subtile $\mathcal{R}_{\sigma}(a)$
does not allow intersection
at the interior points except itself.
Which is tantamount to $\mathcal{T}$ being a tiling (Lemma~\ref{t_criterion}
{\it tiling criterion}).
To see this, we regard the intersection as a diagonal set of the product space
$K_{\sigma}\times K_{\sigma}$, and
perturb (fatten) this `thin' set to apply the weak mixing property in the
symbolic space.

Constraints imposed by algebraic integers are described in terms of {\it polynomials}
(Lemma~\ref{adelicgarsia} {\it Garsia's lemma}), while the Rauzy fractals are defined by {\it power series}. 
As a device to bridge this gap, Section~\ref {specialcylinder}-\ref{poly_like} introduce
special cylinders $\mathcal{C}$ by means of
which it is possible to find polynomial-like behavior in 
the power series 
and obtain the power series version of Garsia's lemma, so to speak (Theorem~\ref{either}).

Section \ref{ergodic}-\ref{distorsion} study
the behavior of the first entry time $\tau_{2}$ to $\mathcal{C}\times\mathcal{C}$
from two aspects. First, $\tau_{2}$ is purely an object defined on the product of the subshift.
On the other hand, it reflects the behavior of polynomial parts of 
Rauzy fractals in $K_{\sigma}\times K_{\sigma}$
by means of Theorem~\ref{either}.
These two aspects are compared to conclude the proof of tiling
(Theorem~\ref{distort}).

A {\it cut and project scheme} is a powerful tool for
understanding the mathematical structure of aperiodic tilings or quasicrystals.
In this modern approach, Rauzy fractals are served as 
a candidate for {\it acceptance windows}. 
Section~\ref{modern} outlines equivalence between tiling and pure discrete spectrum through
this modern point of view.

As a byproduct of a direct proof of tiling, Section~\ref{strongcc}
provides a proof to Strong Coincidence Conjecture.

\section{Substitution}\label{1}

In this section, we collect basic facts about substitution. We refer
to \cite{bernard2000substitution} and
\cite{queffelec2010substitution} for details.

Let ${\mathcal A}$ be a finite alphabet and
denote by ${\mathcal A}^{*}$ the set of
all finite words over ${\mathcal A}$ and the empty set $\emptyset$, which is equipped with
the structure of monoid  by concatenation.
We may set $\mathcal{A}=\{1,2,\dots,n\}$ ($n\geq 2$).
A \textit{substitution} (over ${\mathcal A}$) is a mapping
$\sigma:{\mathcal A}\rightarrow {\mathcal A}^{*}$. 
It is easily seen
that $\sigma$ extends to ${\mathcal A}^{*}$ by requiring $\sigma(w_{1}w_{2})
=\sigma(w_{1})\sigma(w_{2})$ and $\sigma(\emptyset)=\emptyset$.
Let $f:{\mathcal A}^{*}\rightarrow \mathbb{Z}^{n}$ is a homomorphism defined by
\[
f(w_{1}w_{2}\dots w_{k})=\vector{e}_{w_{1}}+\vector{e}_{w_{2}}+
\dots +\vector{e}_{w_{k}},\quad f(\emptyset)={\bf 0}
\]
where $\{\vector{e}_{i}\}_{i=1}^{n}$ denotes the canonical basis of $\mathbb{R}^{n}$.
One can associate to $\sigma$ an {\it incident matrix} $M_{\sigma}$  defined by
\[
M_{\sigma}=[f(\sigma(1))\ f(\sigma(2))\ \cdots\ f(\sigma(n))],
\] 
which is sometimes referred to as the \textit{abelianization} of $\sigma$
due to $f\circ\sigma =M_{\sigma}\circ f$.

A substitution $\sigma$ naturally induces a mapping on $\mathcal{A}^{\mathbb{Z}}$
\[
\sigma(u):=\cdots\sigma(u_{-1}).\sigma(u_{0})\sigma(u_{1})\cdots\quad \text{if}\ u=(u_{k})_{k\in\mathbb{Z}}\in
\mathcal{A}^{\mathbb{Z}}
\]
where . indicates the $0$th coordinate.

The $k$ times iterates of $\sigma$ is denoted by $\sigma^{k}$.
A substitution $\sigma$ is \textit{primitive} if there exists $N>0$ so that
for all $i,j\in{\mathcal A}$ 
the word $\sigma^{N}(i)$ contains $j$.
This is equivalent to the fact that $M_{\sigma}^{N}$
is positive. 
For a primitive substitution $\sigma$, we may assume that 
$\sigma$ has a fixed point $u\in\mathcal{A}^{\mathbb{Z}}$:
$\sigma(u)=u$.
The orbit closure of $u$  
\[
\overline{\mathcal{O}_{\sigma}(u)}=\overline{\{S^{k}u\in\mathcal{A}^{\mathbb{Z}}
: k\in\mathbb{Z}\}}.
\]
and the shift $(S\omega)_{i}=(\omega)_{i+1}$ for $\omega=(\omega_{i})_{i\in\mathbb{Z}}\in
\overline{\mathcal{O}_{\sigma}(u)}$
provide a compact dynamical system
which has a unique $S$-invariant ergodic measure 
$\nu$ and is minimal (strictly ergodic).
The triple $(\overline{\mathcal{O}_{\sigma}(u)},S,\nu)$ is called
the \textit{substitution dynamical system} for $\sigma$.

We call $\sigma$  an \textit{irreducible Pisot substitution}
if the characteristic polynomial
of $M_{\sigma}$ is irreducible over $\mathbb{Q}$ and if
its Perron-Frobenius root $\alpha$ (the maximal eigenvalue) is a Pisot number.
Recall that an algebraic integer $\alpha>1$ is a
\textit{Pisot number} (or \textit{Pisot-Vijayaraghavan number}) if
the other conjugates $\alpha_{2},\dots,\alpha_{n}$
are less than one in modulus.
If $\sigma$ is an irreducible Pisot substitution, it is primitive and any
fixed point of $\sigma:\mathcal{A}^{\mathbb{Z}}\rightarrow\mathcal{A}^{\mathbb{Z}}$
is not $S$-periodic (\cite{holton1998geometric}, hence $\# X_{\sigma}=\infty$).

A \textit{prefix} (resp. \textit{suffix}) of the word $w=w_{1}w_{2}\cdots w_{k}$ is either
a subword of $w$,
$w_{1}w_{2}\cdots w_{i}$ for $1\leq i<k$
(resp. $w_{i}w_{i+1}\cdots w_{k}$ for $1<i\leq k$) or
the empty word $\emptyset$.
By $\vert w\vert_{a}$ denote the number of letter $a$ in $w$.

A word $\sigma^{k}(j)$ has {\it prefix-suffix representations} of
the form $\sigma^{k}(j)=pas$
where $p$ is a prefix, $a\in\mathcal{A}$ and $s$ is a suffix.

%%%
%%%
%%%
{\defn\label{strong}[Strong Coincidence]
A substitution $\sigma$ satisfies the \textit{strong coincidence condition}
if, for every pair $(a,b)\in\mathcal{A}\times\mathcal{A}$, 
there exist $k>0$ and $i\in\mathcal{A}$ so that $\sigma^{k}(a)=
pis$ and $\sigma^{k}(b)=p^{\prime}is^{\prime}$
with $\vert p\vert_{j}=\vert p^{\prime}\vert_{j}$
for every $j\in\mathcal{A}$.
}

\section{Subshift Of Finite Type}\label{2}

For the rest of this paper, $\sigma$ denotes an irreducible Pisot substitution
with the incident matrix $M_{\sigma}$ and its Perron-Frobenius root $\alpha$.

For $a\in\mathcal{A}$ let $\mathrm{Pref}(a)$ denote the set of
prefixes for $\sigma(a)$
\[
\mathrm{Pref}(a)=\{
\emptyset, w_{1}^{(a)},\dots, w_{1}^{(a)}w_{2}^{(a)}\cdots w_{l_{a}-1}^{(a)}\}\quad
\mathrm{if}\ 
\sigma(a)=w_{1}^{(a)}w_{2}^{(a)}\dots w_{l_{a}}^{(a)}.
\]
Notice that $l_{a}=\#\mathrm{Pref}(a)$.
Set $\mathrm{Pref}=\bigcup_{a\in\mathcal{A}}\mathrm{Pref}(a)$.

%%%
\begin{exam}[Rauzy substitution~\cite{rauzy1982nombres}] Let
$\mathcal{A}=\{1,2,3\}$. Define a substitution
$\sigma_{R}$ by the rule $1\mapsto 12,2\mapsto 13,3\mapsto 1$. Then 
${\rm Pref}=\{\emptyset, 1\}$. 
\end{exam}

%%%
%%%
%%%
By $(a:pbs)$ denote the prefix-suffix representation $\sigma(a)=pbs$.

\begin{rem}
For irreducible Pisot substitutions, given $(a:pbs)$, there exists no
other $a^{\prime}\in\mathcal{A}$
so that $\sigma(a^{\prime})=pbs$: otherwise $M_{\sigma}$ would be degenerate.
We may use a shortened form $(a:p)$ for $(a:pbs)$.
\end{rem}

To describe substitution sequences,
Rauzy~\cite{rauzy1990sequences} constructs the \textit{prefix-suffix automaton} for which
the vertex set is $\mathcal{A}$ and  the edge set is $\mathrm{Pref}$. Each edge $p\in
\rm{Pref}$ starts at a vertex $b\in \mathcal{A}$ and terminates at $a\in \mathcal{A}$
if $pa$ is either a prefix of $\sigma(b)$ or $\sigma(b)$ itself. 
Schematically this is shown as $b\xrightarrow[]{p}a$.

In this paper, however, we do not use
the automaton itself. Instead, following \cite{canterini2001automate},
one can associate a topological Markov chain (subshift of finite type) 
with 
the state space $\mathcal{E}$ and the transition matrix
$A=(a_{IJ})_{I,J\in \mathcal{E}}$
to an irreducible Pisot substitution $\sigma$.

%%%
%%%
\begin{defn}
Let $\mathcal{E}$ be the set of all possible prefix-suffix representations for $\sigma$ 
\[
\mathcal{E}=\{ (a:p): a\in\mathcal{A}, p\in\mathrm{Pref}(a) \}.
\]
For $J=(b:pas)$ and $I=(b^{\prime}:p^{\prime}a^{\prime}s^{\prime})$, set
\[
a_{IJ}=
\begin{cases}
1\quad \mathrm{if}\ a=b^{\prime}, \\
0\quad \mathrm{otherwise}. \\
\end{cases}
\]
In other words, $a_{IJ}=1$ when $b\xrightarrow[]{p}b^{\prime}$.
For $A=(a_{IJ})_{I,J\in \mathcal{E}}$,
define 
\[
\Sigma_{A}=\Big\{ \omega=(\omega_{k})_{k\geq 0}
\in\prod_{k=0}^{\infty} \mathcal{E}: a_{\omega_{k}\omega_{k+1}}=1\
\mbox{for all}\ k\geq 0\Big\}
\]
with a metric on it given by, for example,
\[
d_{\Sigma_{A}}((\omega_{k})_{k\geq 0},(\omega^{\prime}_{k})_{k\geq 0})=
\sum_{k=0}^{\infty}\frac{1}{2^{k}}(1-\delta_{\omega_{k}\omega^{\prime}_{k}}),
\quad
\delta_{\omega_{k}\omega^{\prime}_{k}}=
\begin{cases}
1\quad \mathrm{if}\quad \omega_{k}=\omega^{\prime}_{k}, \\
0\quad \mathrm{otherwise}.
\end{cases}
\]
The shift $T:\Sigma_{A}\to\Sigma_{A}$ is defined by $T(\omega_{k})_{k\geq 0}=(\omega_{k+1})_{k\geq 0}$.
Then $\Sigma_{A}$ is a $T$-invariant compact metric space
and is referred to as the {\it subshift of finite type}.
\end{defn}

%%%
%%%
\begin{rem}
In the language of automaton, $\omega=((b_{i+1}:p_{i}a_{i}s_{i}))_{i\geq 0}\in\Sigma_{A}$
is expressed 
in a reverse direction 
\[
\cdots \xrightarrow[]{p_{l}}
b_{l}\xrightarrow[]{p_{l-1}}a_{l-1}=b_{l-1}\xrightarrow[]{p_{l-2}}\cdots
\xrightarrow[]{p_{1}}
a_{1}=b_{1}\xrightarrow[]{p_{0}}a_{0}.
\]
\end{rem}

%%%
%%%
\begin{exam}For Rauzy substitution $\sigma_{R}$, we have
\[ 
\mathcal{E}=\{ (1:\emptyset),(1:1),(2:\emptyset),(2:1),(3:\emptyset) \}\ \mathrm{with}\ \quad D=\# \mathcal{E}=5.
\]
Then the transition matrix $A$ turns out to be
\[
\begin{pmatrix} 1 &0 & 1 & 0 & 1\\
 1 & 0 & 1 & 0 & 1 \\
 0 & 1 & 0 & 0 & 0 \\
 0 & 1 & 0 & 0 & 0 \\
0 & 0 & 0 & 1 & 0
\end{pmatrix}.
\]
\end{exam}

%%%
%%%
%%%
\begin{defn}
Given $l>0$, an {\it admissible path} of length $l$
is a $l$-string of the elements
\[
(b_{1}:p_{0}a_{0}s_{0}),\dots,(b_{l}:p_{l-1}a_{l-1}s_{l-1}) \in \mathcal{E}
\]
with $a_{k}=b_{k}$ for $1\leq k\leq l-1$. 
\end{defn}

%%%
\begin{defn}
Let $p_{1},p_{2}\in\mathcal{A}^{*}$. 
The notation $p_{1}\prec p_{2}$ will mean that $p_{1}$ is a prefix of $p_{2}$.
Similarly $p_{1}\preceq p_{2}$ will be used if and only if $p_{1}$ is a prefix of $p_{2}$
or $p_{2}$ itself.
\end{defn}

Put $A^{k}=(a_{IJ}^{(k)})_{I,J\in \mathcal{E}}$.
It is standard that the number of admissible paths of length $k+1$ 
starting from the vertex
$I$ to $J\in \mathcal{E}$ is given by $a_{IJ}^{(k)}$.

%%%
\begin{lem}\label{irraper}
$A$ is irreducible and aperiodic \rm{(}$A^{N+1}>0$\rm{)}.
\end{lem}
\begin{proof}
For any $J=(j:pas)$ and $I=(i:qbt)$, we need to show $a^{(N+1)}_{IJ}>0$.
Since $\sigma^{N}(a)$ contains $i$ ($M_{\sigma}^{N}>0$), there exists a prefix
$p_{1}$ with $\sigma^{N}(a)\succeq p_{1}i$.
By \cite{dumont1989systemes} and \cite{rauzy1990sequences},
there exists an admissible path $(a_{i+1}:p_{i}a_{i}s_{i})\ (0\leq i\leq N-1)$
so that
\[
p_{1}=\sigma^{N-1}(p_{N-1})\sigma^{N-2}(p_{N-2})\cdots p_{0},\quad a_{0}=i\quad \mathrm{and}
\quad a_{N}=a.
\]
Since it is possible to make transitions from $(a_{N}:p_{N-1}a_{N-1}s_{N-1})$ to $J$
and from $I$ to $(a_{1}:p_{0}a_{0}s_{0})$,
it follows that $a^{(N+1)}_{IJ}>0$.
\end{proof}

\section{Properties Of $A$}

As the characteristic polynomial of $M_{\sigma}$ is irreducible over $\mathbb{Q}$, 
the incident marix $M_{\sigma}$ has only simple eigenvalues.
Let 
\[
(\alpha_{1},\alpha_{2},\ldots,\alpha_{n})
=
(\alpha,\alpha_{2},\ldots,\alpha_{r},\alpha_{r+1},
\overline{\alpha_{r+1}},\ldots,\alpha_{r+s},
\overline{\alpha_{r+s}})\in \mathbb{R}^{r}\times\mathbb{C}^{2s}
\]
be the eigenvalues of $M_{\sigma}$ with $r+2s=n$.
Let $\vector{u}_{i}={}^{t}(u_{1}(i),\dots,u_{n}(i))$ be
an eigenvector of $M_{\sigma}$ for $\alpha_{i}$
\begin{equation}\label{eigen_eq}
M_{\sigma}\vector{u}_{i}=\alpha_{i}\vector{u}_{i}\quad(1\leq i\leq r+s).
\end{equation}
Similarly define $\vector{v}_{i}$ so that 
${}^{t}M_{\sigma}\vector{v}_{i}=\alpha_{i}\vector{v}_{i}$.
It is well-known that $\vector{u}_{1}>0$ and $\vector{v}_{1}>0$
(all positive coordinates) by Perron-Frobenius theorem
and
that $\vector{u}_{i}$ and $\vector{v}_{i}$ can be taken in
$\mathbb{Q}(\alpha_{i})^{n}$, so that
the coordinates of $\vector{u}_{1}$ and $\vector{v}_{1}$ are linearly independent over
$\mathbb{Q}$ respectively.
If we write $\vector{u}_{1}=\vector{u}_{1}(\alpha)$
as a function of $\alpha$, then $\vector{u}_{i}=\vector{u}_{1}(\alpha_{i})$,
and $\vector{v}_{i}=\vector{v}_{1}(\alpha_{i})$ in a similar fashion.
We will specify how to scale $\vector{v}_{1}$ in $\S$\ref{tiling}.
For the moment, we assume $\vector{v}_{1}(\alpha)\in \mathfrak{O}_{K}^{n}$
(i.e. every coordinate of $\vector{v}_{1}$ 
is a polynomial of $\alpha$ with integer coefficients).
No additional rescaling is needed for the unimodular case. 

%%%
\begin{exam}
For Rauzy substitution $\sigma_{R}$, 
\begin{align*}
u_{1}(1)&=1,&u_{2}(1)=\alpha^{2}-\alpha-1,\quad
&u_{3}(1)=-\alpha^{2}+2\alpha, \\
v_{1}(1)&=1,&v_{2}(1)=\alpha-1,\quad
&v_{3}(1)=\alpha^{2}-\alpha-1.
\end{align*}
\end{exam}

Let $l_{a}=\#\mathrm{Pref}(a)$ for $a\in\mathcal{A}$ and let $D=\sum_{a\in\mathcal{A}}l_{a}
=\#\mathcal{E}$.
For $1\leq i\leq r+s$, define 
the vectors $[\vector{u}_{i}]=([\vector{u}_{i}]_{J})_{J\in \mathcal{E}}
\in \mathbb{C}^{D}$ by
\[
[\vector{u}_{i}]_{J}=u_{b}(i)\quad \mathrm{if}\ J=(b:p)\in \mathcal{E}.
\]
By (\ref{eigen_eq})
\[
\sum_{k=1}^{n}\vert \sigma(k)\vert_{b}\cdot u_{k}(i)
=\alpha_{i}u_{b}(i)\quad (b\in\mathcal{A}).
\]
This is equivalent to
\[
\sum_{k=1}^{n}\sum_{J=(k:pbs)} a_{IJ} u_{k}(i)=
\sum_{J\in \mathcal{E}} a_{IJ} [\vector{u}_{i}]_{J}
=\alpha_{i}[\vector{u}_{i}]_{I}\quad (I=(b:*)\in \mathcal{E}),
\]
which means that
$A[\vector{u}_{i}]=\alpha_{i}[\vector{u}_{i}]$ for $1\leq i\leq r+s$.
The eigenvectors $\vector{u}_{i}$ of $M_{\sigma}$ are linearly independent and so are $[\vector{u}_{i}]$.
By construction the transition matrix $A$ has only $n$ linearly independent column vectors.
So $\mathrm{dim Ker}\ A=D-n$ and thus we obtain the decomposition
\begin{equation}\label{decompo}
\mathbb{C}^{D}=\mathrm{Ker}\ A\oplus W_{\alpha_{1}}\oplus\cdots\oplus W_{\alpha_{n}}
\end{equation}
where $W_{\alpha_{i}}$ is the eigenspace of $A$ belonging to $\alpha_{i}$.
This implies that $A$ is diagonalizable.

Similarly define $[\vector{v}_{i}]$
as $[\vector{u}_{i}]$ for $1\leq i\leq r+s$ so that
${}^{t}A[\vector{v}_{i}]=\alpha_{i}[\vector{v}_{i}]$.
Then orthogonal relations between $[\vector{v}_{i}]$ and $[\vector{u}_{j}]$ hold:
if 
\[
\langle \vector{\xi},\vector{\eta}\rangle=\sum_{i=1}^{D}\overline{\xi_{i}}\eta_{i}
\quad(\vector{\xi}=(\xi_{1},\ldots,\xi_{D}),\vector{\eta}=(\eta_{1},\ldots,\eta_{D})),
\]
then
$\langle[\vector{u}_{i}],[\vector{v}_{j}]\rangle=0\ (1\leq i\leq r, i\neq j)$,
\[
\langle[\vector{u}_{i}],[\vector{v}_{j}]\rangle
=
\langle\overline{[\vector{u}_{i}]},[\vector{v}_{j}]\rangle=0\ \mathrm{and}\
\langle[\vector{u}_{i}],[\vector{v}_{i}]\rangle=0
\quad (r+1\leq i\leq r+s, i\neq j).
\]
Furthermore,
\[
\langle \vector{x}_{0},[\vector{v}_{i}]\rangle=\langle \vector{x}_{0},[\overline{\vector{v}_{i}}]\rangle
=0\quad (\vector{x}_{0}\in\mathrm{Ker}\ A, 1\leq i\leq r+s).
\]
%%%
%%%
%%%
\begin{lem}\label{A}
\[
A\vector{x}=\sum_{k=1}^{r}\alpha_{k}
\frac{\langle\overline{\vector{x}},[\vector{v}_{k}]\rangle}{
\langle[\vector{u}_{k}],[\vector{v}_{k}]\rangle}[\vector{u}_{k}]+\sum_{k=r+1}^{r+s}\alpha_{k}
\frac{\langle\overline{\vector{x}},[\vector{v}_{k}]\rangle}{
\langle[\overline{\vector{u}_{k}}],[\vector{v}_{k}]\rangle}[\vector{u}_{k}]
+\overline{\alpha_{k}}
\frac{\langle\overline{\vector{x}},[\overline{\vector{v}_{k}}]\rangle}{
\langle[\vector{u}_{k}],[\overline{\vector{v}_{k}}]\rangle}[\overline{\vector{u}_{k}}].
\]
\end{lem}
\begin{proof}
By~(\ref{decompo}), any $\vector{x}$ has a unique expression $\vector{x}=\vector{x}_{0}+\vector{x}_{1}
+\cdots+\vector{x}_{n}$
where $\vector{x}_{0}\in\mathrm{Ker}\ A$ and $\vector{x}_{k}\in W_{\alpha_{k}}$ for $1\leq k\leq n$.
Put $\vector{x}_{k}=c_{k}[\vector{u}_{k}]$ ($1\leq k\leq r$) and
$\vector{x}_{k}=c_{k}[\vector{u}_{k}],
\vector{x}_{k+1}=c_{k}^{\prime}[\overline{\vector{u}_{k}}]$  ($r+1\leq k\leq r+s$)
for some scalars $c_{k}$
and $c^{\prime}_{k}$. By the orthogonal relations
\[
\langle \vector{x},[\vector{v}_{k}]\rangle=\overline{c_{k}}\langle[\vector{u}_{k}],
[\vector{v}_{k}]\rangle\quad(1\leq k\leq r),\quad
\langle \vector{x},[\vector{v}_{k}]\rangle=\overline{c_{k}^{\prime}}\langle[\overline{\vector{u}_{k}}],
[\vector{v}_{k}]\rangle\quad(r+1\leq k\leq r+s),
\]
and
\[
\langle \vector{x},[\overline{\vector{v}_{k}}]\rangle=\overline{c_{k}}
\langle[\vector{u}_{k}],
[\overline{\vector{v}_{k}}]\rangle\quad(r+1\leq k\leq r+s).
\]
The result follows immediately.
\end{proof}

%%%
%%%
%%%
\begin{thm}\label{asympt}
\[
a_{IJ}^{(k)}=
\alpha^{k}
\frac{[\vector{u}_{1}]_{I}[\vector{v}_{1}]_{J}}{
\langle[\vector{u}_{1}],[\vector{v}_{1}]\rangle}
+
\sum_{i=2}^{r}\alpha_{i}^{k}
\frac{[\vector{u}_{i}]_{I}[\vector{v}_{i}]_{J}}{
\langle[\vector{u}_{i}],[\vector{v}_{i}]\rangle}+\sum_{i=r+1}^{r+s}\alpha_{i}^{k}
\frac{[\vector{u}_{i}]_{I}[\vector{v}_{i}]_{J}}{
\langle[\overline{\vector{u}_{i}}],[\vector{v}_{i}]\rangle}
+\overline{\alpha_{i}}^{k}
\frac{\overline{[\vector{u}_{i}]_{I}}\overline{[\vector{v}_{i}]_{J}}}{
\langle[\vector{u}_{i}],[\overline{\vector{v}_{i}}]\rangle}.
\]
\end{thm}
\begin{proof}
Since $a_{IJ}^{(k)}={}^{t}\vector{e}_{I}A^{k}\vector{e}_{J}$, use Lemma~\ref{A}.
\end{proof}

\section{Parry Measure And Weak Mixing  Subshift}\label{pmeasure}

%%%
%%%
\begin{defn}
Given an admissible path $c_{0},c_{1},\ldots,c_{k}\in \mathcal{E}$,
define a \textit{cylinder} $\langle c_{0}c_{1}\cdots c_{k}\rangle$ of $\Sigma_{A}$ 
by
\[
\langle c_{0}c_{1}\cdots c_{k}\rangle=
\{ ((b_{k+1}:p_{k}))_{k\geq 0}\in\Sigma_{A}: (b_{i+1}:p_{i})=c_{i}\ \mathrm{for}\
0\leq i\leq k \}.
\]
\end{defn}

The \textit{Parry measure} $m$ on $\Sigma_{A}$ is a $T$-invariant
$(\vector{p},P)$-Markov measure as follows (\cite{parry1964intrinsic}).
Define the probability row vector $\vector{p}=(p_{I})_{I\in \mathcal{E}}$
and the stochastic matrix $P=(p_{IJ})_{I,J\in \mathcal{E}}$ by
\[
p_{I}=\frac{[\vector{u}_{1}]_{I}[\vector{v}_{1}]_{I}}{
\langle[\vector{u}_{1}],[\vector{v}_{1}]\rangle}
,\quad 
p_{IJ}=a_{IJ}\frac{[\vector{u}_{1}]_{J}}{\alpha[\vector{u}_{1}]_{I}}.
\]
Then one can make a finitely additive measure on the algebra $\mathfrak{A}$
of finite unions of disjoint cylinders which is 
defined by $m(\langle c_{0}\rangle)=p_{c_{0}}$ and
\[
m(\langle c_{0}c_{1}\cdots c_{k} \rangle)=p_{c_{0}}p_{c_{0}c_{1}}\cdots p_{c_{k-1}c_{k}}>0
\quad\mathrm{for\ any\ cylinder}\ \langle c_{0}c_{1}\cdots c_{k}\rangle,
\]
and extend it to a probability measure $m$ on the $\sigma$-algebra generated by $\mathfrak{A}$.
By $\vector{p}P=\vector{p}$ follows $T$-invariance.

As $A$ is irreducible and aperiodic by Lemma~\ref{irraper}, 
so is $P$. Then it is known that {\it the measure-preserving system $(\Sigma_{A},T,m)$ is 
 weak-mixing} (in fact, strong-mixing. See
Theorem~1.31 of \cite{walters2000introduction}).
Equivalently {\it the product system $(T\times T,\Sigma_{A}\times\Sigma_{A},m\times m)$ is ergodic}
where $T\times T(\omega_{1},\omega_{2})=(T\omega_{1},T\omega_{2})$ (see Theoem~1.24 of \cite{walters2000introduction}, and
\cite{furstenberg2014recurrence} for example).
So, if $\mathcal{C}$ is a non-empty cylinder set ($m(\mathcal{C})>0$), then
almost every point $(\omega_{1},\omega_{2})$ in $\Sigma_{A}\times\Sigma_{A}$
visits $\mathcal{C}\times\mathcal{C}$ infinitely often by $T\times T$ (
Birkhoff's ergodic theorem).

\section{Rauzy Fractals From The Adelic Viewpoint}

Following the exposition of \cite{minervino2014geometry} and \cite{sing2006pisot}
, a geometric representation of
$\Sigma_{A}$ will be introduced.
Since this description is closely related to the ring of 
integers $\mathfrak{O}_{K}$ in the number field
$K=\mathbb{Q}(\alpha)$ (cf. \cite{thurston1989groups}),
we first
recall  facts from algebraic number theory.

The number field $K$ has $r$ real embeddings (field homomorphisms) and $2s$ complex ones.
A \textit{prime} of $K$ is an equivalent class of valuations.
To each prime ideal, each real embedding and
each conjugate pair of complex embeddings, there corresponds exactly one prime $v$ of $K$ and
vice versa (Ostrowski's theorem).
If a prime $v$ corresponds to a prime ideal, then we call it a {\it finite prime}.
For an {\it infinite prime} $v$, it means that $v$ comes from an equivalent class of real or complex
embedding of $K$. Let $\mathfrak{M}_{\infty}$ be the set of infinite primes of $K$.
If $v\in\mathfrak{M}_{\infty}$ and if $\tau_{v}:K\to\mathbb{R}$ is a real embedding
corresponding to $v$, 
set $K_{v}=\mathbb{R}$ and define 
an absolute value (multiplicative valuation) $\vert\cdot\vert_{v}: K\to\mathbb{R}$
by $\vert \xi\vert_{v}=\vert \tau_{v}(\xi)\vert$.
If $\tau_{v}:K\to\mathbb{C}$ is a complex embedding corresponding to $v$,
put $K_{v}=\mathbb{C}$ and define $\vert \xi\vert_{v}=\vert \tau_{v}(\xi)\vert^{2}$.
If $v$ is a finite prime corresponding to a prime ideal $\mathfrak{p}$,
denote the absolute norm of $\mathfrak{p}$
by $\mathfrak{N}(\mathfrak{p})$ and the corresponding {\it $\mathfrak{p}$-adic
valuation} by $v_{\mathfrak{p}}$. 
Then the 
{\it normalized absolute value} (multiplicative valuation) is defined by
\[
\vert \xi\vert_{v}=\Big(\frac{1}{\mathfrak{N}(\mathfrak{p})}\Big)^{v_{\mathfrak{p}}
(\xi)}\quad(\xi\in K,\ \vert 0\vert_{v}=0),
\]
with respect to which the completion of $K$ is denoted by $K_{v}$.

Let $(\alpha)$ be the principal ideal generated by $\alpha$.
By the unique factorization of ideals, it follows that
\begin{equation}\label{ideal}
(\alpha)=\mathfrak{p}_{1}^{\nu_{1}}\mathfrak{p}_{2}^{\nu_{2}}\dots\mathfrak{p}_{k}^{\nu_{\kappa}}
\quad (\nu_{i}\geq 1,\nu_{i}\in\mathbb{N})
\end{equation}
with the $\mathfrak{p}_{i}$ distinct prime ideals for $1\leq i\leq \kappa$.
If $N_{K/{\mathbb{Q}}}(\xi)$ is the field norm of $K$, then
$\mathfrak{N}((\alpha))=\vert N_{K/{\mathbb{Q}}}(\alpha)\vert$.
Since $\mathfrak{N}((\alpha))=\prod_{i=1}^{k}
\mathfrak{N}(\mathfrak{p}_{i})^{\nu_{i}}$ by (\ref{ideal}),
and since $\vert N_{K/{\mathbb{Q}}}
(\alpha)\vert=\vert{\rm det}M_{\sigma}\vert$,
we obtain 
%%%
%%%  
%%%
\begin{equation}\label{product_alpha}
\vert\alpha_{2}\cdots\alpha_{d}\vert\frac{1}{q_{1}^{\nu_{1}}}
\cdots\frac{1}{q_{k}^{\nu_{\kappa}}}=\frac{1}{\alpha}
\end{equation}
where $q_{i}=\mathfrak{N}(\mathfrak{p}_{i})$ and each $q_{i}$ is a power of some prime number.
Applying the \textit{product formula}
\[
\prod_{v}\vert \xi\vert_{v}=1\qquad (\xi\in K^{\times})
\]
to $\alpha$ also gives the same result, where the product is taken over all the primes $v$ of $K$.

Let  $\mathfrak{M}^{\prime}$ be the union of $\mathfrak{M}_{\infty}$
and the subset of finite primes which correspond to $\mathfrak{p}_{i}$
for $1\leq i\leq \kappa$.
Set $\mathfrak{M}=\mathfrak{M}^{\prime}\backslash\{ v_{1}\}$ where  $v_{1}$ is
such a valuation as $\vert\alpha\vert_{v_{1}}=\alpha$ (expanding direction).
Define the ad\'{e}le subrings by 
\[
K_{\infty}=\prod_{v\in\mathfrak{M}_{\infty}}K_{v}=
\mathbb{R}^{r}\times\mathbb{C}^{s}
,\quad  K_{\alpha}=\prod_{v\in\mathfrak{M}^{\prime}}K_{v}\quad
\mathrm{and}\quad
K_{\sigma}=\prod_{v\in\mathfrak{M}}K_{v}.
\]
Observe that
$M_{\sigma}$ is unimodular if and only if $(\alpha)=\mathfrak{O}_{K}$.
So, in the unimodular case, there is no prime ideal $\mathfrak{p}$ which divides
$(\alpha)$, i.e.$(\alpha)\subseteq\mathfrak{p}$,
hence $\mathfrak{M}_{\infty}=\mathfrak{M}^{\prime}$.

A metric $d_{K}$ on $K_{\sigma}$ can be introduced, for instance, by
\[
d_{K}(X,Y)=\max\Big\{ \vert X_{v}-Y_{v}\vert,\vert X_{v^{\prime}}-Y_{v^{\prime}}\vert_{v^{\prime}}
: v\in\mathfrak{M}\cap \mathfrak{M}_{\infty},v^{\prime}\in \mathfrak{M}\backslash\mathfrak{M}_{\infty}
\Big\}
\]
for $X=(X_{v})_{v\in\mathfrak{M}}$ and  $Y=(Y_{v})_{v\in\mathfrak{M}}$.
For a metric space, we denote a ball of radius $R$ at $X$ by $B(X,R)$.

Let $\tau_{v}:K\to\mathbb{C}$ be the embedding corresponding to $v\in
\mathfrak{M}_{\infty}$. For $X\in K_{\alpha}$ write $X=(X_{v})_{v\in\mathfrak{M}^{\prime}}$.
Define $\Phi:K\to K_{\alpha}$ (`diagonal embedding') by
\[
\Phi(\xi)_{v}=
\begin{cases}
\tau_{v}(\xi) &\mathrm{if}\ v\in\mathfrak{M}_{\infty}, \\
\xi  &\mathrm{if}\ v\in\mathfrak{M}^{\prime}\backslash\mathfrak{M}_{\infty}.
\end{cases}
\]
By $\pi_{2}$ denote the projection
from $K_{\alpha}$ to $K_{\sigma}$: $\pi_{2}((X_{v})_{v\in\mathfrak{M}^{\prime}})
=(X_{v})_{v\in\mathfrak{M}}$.
Similarly, by $\pi_{1}$ denote the projection from
$K_{\alpha}$ to $\mathbb{R}$:
$\pi_{1}((X_{v})_{v\in\mathfrak{M}^{\prime}})=(X_{v})_{v\in\mathfrak{M}^{\prime}
\backslash \mathfrak{M}}$.

The action of $K$ on $K_{\sigma}$ by multiplication
is interpreted as
\[
\beta(X_{v})_{v\in\mathfrak{M}}=
(\Phi(\beta)_{v}X_{v})_{v\in\mathfrak{M}}\quad\mathrm{for}\ 
\beta\in K\ \mathrm{and}\ (X_{v})_{v\in\mathfrak{M}}\in K_{\sigma}.
\]
In particular, $\beta(\pi_{2}\circ\Phi(\xi))=
\pi_{2}\circ\Phi(\beta \xi)$. 

Let $\mu$ be the Haar measure on $K_{\sigma}$ (translation-invariant measure).
For any measurable set $B$ in $K_{\sigma}$
\begin{equation}\label{haar}
\mu(\alpha B)=\prod_{v\in\mathfrak{M}}
\vert\alpha\vert_{v}\mu(B)=
\frac{1}{\alpha}\mu(B).
\end{equation}

By abuse of notation, $\langle\cdot,\cdot\rangle$
is also used as the inner product in  $\mathbb{C}^{n}$.
Define
$\Psi:\Sigma_{A}\rightarrow K_{\sigma}
$ by
%%%
%%%
%%%
\[
\Psi((a_{k+1}:p_{k})_{k\geq 0})
=\Big(\sum_{i\geq 0}\langle f(p_{i}),\vector{v}(\alpha_{v})\rangle \alpha^{i}_{v}
\Big)_{v\in\mathfrak{M}}
\]
where $\alpha_{v}=\Phi(\alpha)_{v}$ for $v\in\mathfrak{M}$.
Notice that $\sum_{i\geq 0}\langle f(p_{i}),\vector{v}(\alpha_{v})\rangle \alpha^{i}_{v}$
is well-defined
in every $K_{v}$. Indeed, it is obvious when $v\in\mathfrak{M}_{\infty}\backslash\{
v_{1}\}$.
If $v\in\mathfrak{M}\backslash\mathfrak{M}_{\infty}$, then
$v_{\frak{p}_{i}}(\alpha)=\nu_{i}$ and $v_{\frak{p}_{i}}(\alpha^{k}
)=k\nu_{i}$ for $1\leq i\leq \kappa$.
Thus
$\eta=\sum_{i=0}^{\infty}\langle f(p_{i}),
\vector{v}(\alpha)\rangle\alpha^{i}$ is the limit of a Cauchy sequence
\[
\Big\{\sum_{i=0}^{k}\langle f(p_{i}),
\vector{v}(\alpha)\rangle\alpha^{i}\Big\}_{k\in\mathbb{N}}\subset \mathfrak{O}_{K}
\]
with respect to the metric of $K_{v}$ and thereby $\vert\eta\vert_{v}\leq 1$.

%%%
%%%  def. of Rauzy fractal in general
%%%
\begin{defn}\label{rauz}
Given $a\in\mathcal{A}$ let
\[
\Sigma_{A}(a)=\Big\{ \omega=(\omega_{k})_{k\geq 0}\in\Sigma_{A}
: \omega_{0}=(b_{1}:p_{0}a_{0}s_{0})\ \mathrm{with}\ a_{0}=a
\Big\}.
\]
The image of $\Sigma_{A}$ by $\Psi$
\[
\mathcal{R}_{\sigma}=\Psi(\Sigma_{A})
\]
is called the \textit{Rauzy fractal} for $\sigma$
(or \textit{Dumont-Thomas central tile}),
and $\mathcal{R}_{\sigma}(a)=\Psi(\Sigma_{A}(a))$
is referred to as its \textit{subtiles} (or \textit{
Dumont-Thomas subtiles}).
\end{defn}

%%%
\begin{rem}\label{identity}
This definition of Rauzy fractal and its subtiles might not seem standard, but
it actually aligns with the conventional definition as seen below. 
Define $T^{-1}_{\rm ext}:K_{\alpha}\times\mathcal{A}\rightarrow 2^{K_{\alpha}\times\mathcal{A}}$ by
\[
T^{-1}_{\rm ext}(Y^{*},a)=\bigcup_{(b:p,a,s)\in \mathcal{E}}\{(\alpha^{-1}(Y^{*}+
\Phi(\langle f(p),\vector{v}\rangle
)),b)\}.
\]
Iterating $T^{-1}_{\rm ext}$, we obtain
%%% 
%%%
%%%
\begin{multline}\label{iterate}
T^{-k}_{\rm ext}(\vector{0},a)=
\bigcup_{\sigma^{k}(b)=pas}\{ (\alpha^{-k}
\Phi(\langle f(p),\vector{v}\rangle
),b)\}\\
=\bigcup_{
b_{k}\xrightarrow[]{p_{k-1}}\cdots
\xrightarrow[]{p_{1}}
b_{1}\xrightarrow[]{p_{0}}a
}
\{ (\alpha^{-k}
\Phi(\langle f(\sigma^{k-1}(p_{k-1})\sigma^{k-2}(p_{k-2})\dots p_{0}),\vector{v}\rangle
),b_{k})\}
\end{multline}
where the sum is taken over all the admissible paths of length $k$ which
end at $a$.
The Hausdorff metric $d_{\mathrm{H}}$ for closed sets $A,B\in K_{\sigma}$ is defined by
\[
d_{\mathrm{H}}(A,B)=\sup\{ \rho(A,B),\rho(B,A)\}
\]
where $\rho(A,B)=\sup_{x\in A}d(x,B)$ and $d(x,B)=\inf_{y\in B}d_{K}(x,y)$.
In (4.12) or (8.1) of \cite{minervino2014geometry}, the Dumont-Thomas subtile
is described by the limit with respect to $d_{\mathrm{H}}$
\[
\mathcal{R}_{a}:=\underset{k\to\infty}{\mathrm{lim}_{\mathrm{H}}}\
\alpha^{k}\pi\circ T^{-k}_{\mathrm{ext}}(\vector{0},a)
\]
where $\pi$ modules out the 2nd coordinate and makes a composition of $\pi_{2}$ and the 1st coordinate
into $K_{\sigma}$.
For our purpose, it suffices to show $d_{\mathrm{H}}(\mathcal{R}_{\sigma}(a),\mathcal{R}_{a})=0$.
By the triangle inequality, we obtain
%%%
%%%
\begin{equation}\label{1sttriangle}
d_{\mathrm{H}}(\mathcal{R}_{\sigma}(a),\mathcal{R}_{a})\leq d_{\mathrm{H}}(\mathcal{R}_{\sigma}(a),
\alpha^{k}\pi\circ T^{-k}_{\mathrm{ext}}(\vector{0},a))
+d_{\mathrm{H}}(\alpha^{k}\pi\circ T^{-k}_{\mathrm{ext}}(\vector{0},a),\mathcal{R}_{a}).
\end{equation}
Given $\epsilon>0$, there exists $k_{0}$ such that for all $k\geq k_{0}$
%%%
%%%
\begin{equation}\label{2ndtriangle}
d_{\mathrm{H}}(\alpha^{k}\pi\circ T^{-k}_{\mathrm{ext}}(\vector{0},a),\mathcal{R}_{a})
<\frac{\epsilon}{2}.
\end{equation}
By (\ref{iterate}), the set $\alpha^{k}\pi\circ T^{-k}_{\mathrm{ext}}(\vector{0},a)$ is determined
by admissible paths 
\[
b_{k}\xrightarrow[]{p_{k-1}}b_{k-1}\xrightarrow[]{p_{k-2}}\cdots
\xrightarrow[]{p_{1}}
b_{1}\xrightarrow[]{p_{0}}a,\quad p=f(\sigma^{k-1}(p_{k-1}))\cdots f(p_{1})p_{0}.
\] 
Since
\[
\rho(\mathcal{R}_{\sigma}(a),\alpha^{k}\pi\circ T^{-k}_{\mathrm{ext}}(\vector{0},a)
)=\sup_{X\in\mathcal{R}_{\sigma}(a)}
d(X,\alpha^{k}\pi\circ T^{-k}_{\mathrm{ext}}(\vector{0},a))
\]
is bounded by a constant multiplied by $\beta^{k}$ for some $0<\beta<1$,
it follows by (\ref{1sttriangle}) and (\ref{2ndtriangle})
that 
$d_{H}(\mathcal{R}_{\sigma}(a),\mathcal{R}(a))<\epsilon$
for sufficiently large $k$. Hence $\mathcal{R}_{\sigma}(a)=\mathcal{R}(a)$.
{\it This allows for the utilization of existing knowledge about Rauzy fractals.}
\end{rem}

%%%
%%%
\begin{lem}\label{continuous}
$\Psi:\Sigma_{A}\to K_{\sigma}$ is continuous and so is its translation
$\Psi+\gamma$ for $\gamma\in K_{\sigma}$.
\end{lem}
\begin{proof}
Fix $\omega\in\Sigma_{A}$. 
Set 
\[
C_{i}:=\max\{\vert\langle f(p)-f(q),\vector{v}_{i}\rangle\vert:
p,q\in\mathrm{Pref}\}
\]
for $2\leq i\leq r+s$. Given $\epsilon>0$, take the minimum integer $k$ so that
\[
\max\Big\{
C_{i}\frac{\vert\alpha_{i}\vert^{k}}{1-\vert\alpha_{i}\vert},
\Big(\frac{1}{q_{j}}\Big)^{k\nu_{j}}:
2\leq i\leq r+s,1\leq j\leq \kappa
\Big\}
<\epsilon.
\]
If $d_{\Sigma_{A}}(\omega,\omega^{\prime})<2^{-(k-1)}$, then
the first $k$ coordinates of $\omega$ and $\omega^{\prime}$ coincide.
For $v\in\mathfrak{M}_{\infty}$,
it is easy to see $\vert\Psi(\omega)_{v}-\Psi(\omega^{\prime})_{v}\vert<\epsilon$.
For $v\in \mathfrak{M}\backslash\mathfrak{M}_{\infty}$, one can find $\eta\in K_{v}$
so that $\Psi(\omega)_{v}-\Psi(\omega^{\prime})_{v}=\alpha^{n}\eta$ with $\vert\eta\vert_{v}\leq 1$.
This implies $\vert\Psi(\omega)_{v}-\Psi(\omega^{\prime})_{v}\vert_{v}<\epsilon$.
Consequently we have
\[
\Psi\Big(B\Big(\omega,\frac{1}{2^{k-1}}\Big)\Big)\subset B(\Psi(\omega),\epsilon).
\]
\end{proof}

As any closed ball at $0$ in 
every locally compact space $K_{v}$ is compact,
so is $\overline{B(0,R)}$ in $K_{\sigma}$ for $R>0$. 
Hence, in view of Lemma~\ref{continuous}, the Rauzy fractal $\mathcal{R}_{\sigma}$
is compact.

There are a sequence of partitions of $\Sigma_{A}(a)$ induced by cylinders:
for each $m\geq 1$
\begin{equation}\label{parti}
\Sigma_{A}(a)=\bigcup\langle c_{0}c_{1}\cdots c_{m-1}\rangle
\end{equation}
where the union is taken over all the different admissible paths of length $m$ 
starting with $c_{0}=
(b_{1}:p_{0}as_{0})$ for all possible $b_{1}\in\mathcal{A}$ and $p_{0}\in\mathrm{Pref}(b_{1})$.

If $c_{i}=(b_{i+1}:p_{i}a_{i}s_{i})$ for $0\leq i\leq m-1$ and $a_{0}=a$, then
\[
\sigma^{m}(b_{m})=pas, \quad p=\sigma^{m-1}(p_{m-1})\sigma^{m-2}(p_{m-2})\cdots p_{0}
\]
for some suffix $s$, and
\begin{equation}\label{cylinder_image}
\Psi(\langle c_{0}c_{1}\cdots c_{m-1}\rangle)=\alpha^{m}\mathcal{R}_{\sigma}(b_{m})
+\Phi^{\prime}(\langle f(p),\vector{v}\rangle),\quad
\Phi^{\prime}=\pi_{2}\circ\Phi.
\end{equation}
Each partition (\ref{parti}) and (\ref{iterate}), thus, induce the set equation 
%%%
%%%
%%%
\begin{equation}\label{seteqI}
\mathcal{R}_{\sigma}(a)=
\bigcup_{(\gamma^{*},b)\in T^{-m}_{\rm ext}({\bf 0},a)}\alpha^{m}(\mathcal{R}_{\sigma}(b)+
\gamma^{*})
=
\bigcup_{b\in\mathcal{A},\sigma^{m}(b)=pas}
\alpha^{m}\mathcal{R}_{\sigma}(b)+\Phi^{\prime}(\langle f(p),\vector{v}\rangle)
\end{equation}
for each $m\geq 1$, where the union members are disjoint in measure by Perron-Frobenius theorem
(Theorem~8.3 of \cite{minervino2014geometry}).
Moreover, it is well-known for $\mathcal{R}_{\sigma}(a)$ to have the following properties
%%%
%%%
\begin{equation}\label{ca}
\mu(\partial\mathcal{R}_{\sigma}(a))=0\ \mathrm{and}\
\mathcal{R}_{\sigma}(a)=\overline{\mathrm{int}\mathcal{R}_{\sigma}(a)}\
\mathrm{for\ all}\ a\in\mathcal{A}.
\end{equation}

%%%
%%%
%%%
\begin{lem}\label{latticeII}
Let $\vector{w}_{1}$ and $\vector{w}_{2}$ be vectors in $\mathbb{Q}^{n}$. Then
$\vector{w}_{1}=\vector{w}_{2}$ if and only if
\[
\Phi^{\prime}(\langle \vector{w}_{1},\vector{v}\rangle)=
\Phi^{\prime}(\langle \vector{w}_{2},\vector{v}\rangle).
\]
\end{lem}
\begin{proof}
For $\vector{w}\in\mathbb{Q}^{n}$, suppose that
$\Phi^{\prime}(\langle\vector{w},\vector{v}\rangle)=0$.
It is equivalent to $\langle \vector{w},\vector{v}(\alpha_{v})\rangle=0$
for all $v\in\mathfrak{M}^{\prime}$. This is true even if $M_{\sigma}$ is unimodular,
because field homomorphisms from $\mathbb{Q}(\alpha)$ to $\mathbb{Q}(\alpha_{i})$
are injective.
By the decomposition associated with $M_{\sigma}$
\[
\vector{w}=\sum_{i=1}^{r}\frac{\langle\overline{\vector{w}},\vector{v}_{i}\rangle}
{\langle\vector{u}_{i},\vector{v}_{i}\rangle}
\vector{u}_{i}+
\sum_{i=r+1}^{r+s}\frac{\langle\overline{\vector{w}},\vector{v}_{i}\rangle}
{\langle\overline{\vector{u_{i}}},\vector{v}_{i}\rangle}\vector{u}_{i}+
\frac{\langle\overline{\vector{w}},\overline{\vector{v}_{i}}\rangle}
{\langle\vector{u_{i}},\overline{\vector{v}_{i}}\rangle}\overline{\vector{u}_{i}},
\]
it follows that $\vector{w}=0$, and the converse is obvious.
\end{proof}

\section{Multiple Tiling Of $K_{\sigma}$}\label{tiling}

Let $Z=\bigcup_{i\geq 0}M_{\sigma}^{-i}\mathbb{Z}^{n}$ and set $\vector{v}=
\vector{v}(\alpha)=\vector{v}_{1}$.
%%%
%%%
\begin{defn}
The \textit{translation set} $\Gamma$ is defined by
\[
\Gamma=\{(\Phi^{\prime}(\langle\vector{w},\vector{v}\rangle),a)\in K_{\sigma}\times\mathcal{A}
:\vector{w}\in Z,
\langle\vector{w},\vector{v}\rangle\geq 0,
\langle\vector{w}-\vector{e}_{a},\vector{v}\rangle< 0 \}.
\]
\end{defn}

%%%
\begin{rem}
Let $\vector{v}=(v_{1},\ldots,v_{n})$ and consider the $\mathbb{Z}$-module
$V=\langle v_{1},\ldots,v_{n}\rangle_{\mathbb{Z}}$.
As mentioned in comments after Theorem~7.3 of \cite{minervino2014geometry},
the translation set $\Gamma$ defined above is the same as
\[
\{ (\Phi^{\prime}(w),a)\in K_{\sigma}\times\mathcal{A}:
 w\in V\cdot\mathbb{Z}[\alpha^{-1}]\cap[0,\langle \vector{e}_{a},\vector{v}\rangle) \}.
\]
To put it more precisely, 
\begin{itemize}
\item Every $v_{i}$ belongs to $q^{-1}\mathbb{Z}[\alpha]$ for some positive integer $q$, and therefore
$V\cdot\mathbb{Z}[\alpha^{-1}]=\mathbb{Z}[\alpha^{-1}]v_{1}+
\cdots+\mathbb{Z}[\alpha^{-1}]v_{n}$ is a fractional ideal of $\mathbb{Z}[\alpha^{-1}]$
(\S 3 of \cite{minervino2014geometry}),
\item
$\langle M_{\sigma}^{-k}\vector{w},\vector{v}\rangle=\alpha^{-k}
\langle \vector{w},\vector{v}\rangle$ for $k\in\mathbb{Z}$,
\item The isomorphism 
between $Z$ and $V\cdot\mathbb{Z}[\alpha^{-1}]$ is given by 
$\vector{w}\mapsto\langle \vector{w},\vector{v}\rangle$
(Lemma~7.1 of \cite{minervino2014geometry}).
\end{itemize}

The reason why we prefer our definition of $\Gamma$ is
that it will be essential for our discussion in $\S\ref{distorsion}$ .
\end{rem}

%%%
\begin{rem}
By Lemma~\ref{latticeII},
each $(\gamma,a)\in \Gamma$ corresponds to a unique $\vector{w}\in Z$
so that $\gamma=\Phi^{\prime}(\langle \vector{w},\vector{v}\rangle)$.
\end{rem}

Let $\mathcal{T}=\{ T_{i}\}_{i}$ be a collection of compact subsets of $K_{\sigma}$
with $\mu(\partial T_{i})=0$ for every $T_{i}$.
Define the {\it covering degree} at $X\in K_{\sigma}$ by
\[
d_{\mathrm{cov}}(X)=\#\{ T_{i}\in\mathcal{T}:X\in T_{i}\}
\]
(see Definition~5.65 of \cite{sing2006pisot}).
We refer to $\mathcal{T}$ as a {\it multiple tiling} if 
each $T_{i}\in\mathcal{T}$ is the closure of its interior 
and if there exists a positive integer $d_{\mathrm{cov}}\geq 1$ so that
$d_{\mathrm{cov}}(X)=d_{\mathrm{cov}}$ for $\mu$-almost every $X\in K_{\sigma}$.
If $d_{\mathrm{cov}}=1$, then $\mathcal{T}$ is called a {\it tiling}.

%%%
%%%
\begin{lem}[tiling criterion]\label{t_criterion}
Let $\mathcal{T}$ be a  multiple tiling.
The followings are equivalent.
\begin{itemize}
\item[(1)]$\mathcal{T}$ is a {\it tiling}.
\item[(2)]There exists one member $T_{i}\in\mathcal{T}$ so that
$\mathrm{int}T_{i}\cap\mathrm{int}T_{j}=\emptyset$ for $i\neq j$.
\end{itemize}
\end{lem}
\begin{proof}
Observe that if $\mathrm{int}T_{i}\cap\mathrm{int}T_{j}\neq\emptyset$ ($i\neq j$),
there is an open set of $T_{i}$  for which  $d_{\mathrm{cov}}(X)\geq 2$,
and hence (1) implies (2). Assume (2). Then $d_{\mathrm{cov}}(X)=1$ 
for an open set of positive measure
in $T_{i}$. Since $d_{\mathrm{cov}}(X)=d_{\mathrm{cov}}$ a.e. 
in the case of multiple tiling,
it follows that $d_{\mathrm{cov}}=1$, which implies (1).
\end{proof}

Let
\[
\mathcal{T}=\{\mathcal{R}_{\sigma}(a)+\gamma: (\gamma, a)\in\Gamma\}.
\]

%%%
%%%  Minervino-Thuswaldner
%%%
\begin{thm}[Theorem~9.2 of \cite{minervino2014geometry},
\cite{ito2006atomic} and \cite{barge2006geometric}]\label{mt}
Let $\sigma$ be an irreducible Pisot substitution. Then
$\mathcal{T}$
is a multiple tiling of $K_{\sigma}$.
\end{thm}

A subset $W$ of $K_{\sigma}$ is a {\it Delone set} if it is relatively dense and uniformly discrete.
Equivalently $W$ is a Delone set if and only if there exist $r_{1},r_{2}>0$
so that 
\[
\# (W\cap\overline{B(X,r_{1}
)})\geq 1\quad\mathrm{and}\quad
\# (W\cap B(X,r_{2}))\leq 1
\]
for all $X\in K_{\sigma}$.
Lemma~6.6 of \cite{minervino2014geometry} asserts that
a subset of $K_{\sigma}$ derived from $\Gamma$
%%%
\begin{equation}\label{derived}
\Gamma_{a}:=\{ \gamma=\Phi^{\prime}(\langle\vector{w},\vector{v}\rangle)
:\vector{w}\in Z,
\langle\vector{w},\vector{v}\rangle\geq 0,
\langle\vector{w}-\vector{e}_{a},\vector{v}\rangle< 0 \}
\end{equation}
is a Delone set (inter model set) for each $a\in\mathcal{A}$.

%%%
%%%
\begin{lem}
There are only a finite number of $(\gamma,a)\in\Gamma$
for which 
\begin{equation}\label{f_intersect}
\mathcal{R}_{\sigma}\cap (\mathcal{R}_{\sigma}(a)+\gamma)\neq\emptyset.
\end{equation}
\end{lem}
\begin{proof}
Notice that the diameter of $\mathcal{R}_{\sigma}(a)$ is bounded
for every $a$.
Then one can choose $R>0$ so that $\mathcal{R}_{\sigma}$ is contained in a ball $B(0,R)$
and that if $\gamma\not\in B(0,R)$, $\mathcal{R}_{\sigma}(a)+\gamma$
does not intersect with $\mathcal{R}_{\sigma}$ for $(\gamma,a)\in\Gamma$.

Given $a\in \mathcal{A}$,
since $\Gamma_{a}$ in (\ref{derived})
is a Delone set, there exists $r_{2}>0$ so that
$\#(\Gamma_{a}\cap B(X,r_{2}))\leq 1$ for all $X\in K_{\sigma}$.
Cover the closed ball $\overline{B}(0,R)$ with balls of radius $r_{2}$.
Since $\overline{B}(0,R)$ is compact,
one can take a finite number of balls of radius $r_{2}$ which cover $\overline{B}(0,R)$.
So the number of $(\gamma,a)$ satisfying (\ref{f_intersect}) is
less than this finite number. Since $\#\mathcal{A}=n$, the proof completes.
\end{proof}

If we rescale $\vector{v}$ by $c\vector{v}$ 
(a scalar multiplication $c>0$), the Rauzy fractal $\mathcal{R}_{\sigma}$,
its subtile $\mathcal{R}_{\sigma}(a)$ and any 
translation $\gamma$ for $(\gamma,a)\in\Gamma$
get transformed accordingly into
$c\mathcal{R}_{\sigma}$,$c\mathcal{R}_{\sigma}(a)$
 and $c\gamma$ respectively. Yet
a finite subset
\[
Z_{0}=\{ \vector{w}\in Z:
\mathcal{R}_{\sigma}\cap (\mathcal{R}_{\sigma}(a)+\gamma)\neq\emptyset,
\quad\gamma=
\Phi^{\prime}(\langle\vector{w},\vector{v}\rangle)\ \mathrm{for}\ (\gamma,a)\in\Gamma \}
\]
still remains the same for any scalar multiple  of $\vector{v}$.
Therefore {\it one can take a common integer $c>0$ so that
$\langle\vector{w},c\vector{v}\rangle\in\mathfrak{O}_{K}$
for all $\vector{w}\in Z_{0}$}. As mentioned earlier, when $M_{\sigma}$ is  unimodular,
this prescription is not needed because of $Z=\mathbb{Z}^{n}$. 
In the sequel, we will adopt this scaling $c\vector{v}$ and denote it by $\vector{v}$.

\section{Perturbation}\label{perturbation}

First we will prove that $\mathcal{T}$ is a tiling ($d_{\mathrm{cov}}=1$).
Since the proof is long, it will be convenient to divide it into several sections 
(Section~\ref{perturbation}-\ref{distorsion}).
Suppose that
\[
\mathrm{int}\mathcal{R}_{\sigma}(a)\cap\mathrm{int}
(\mathcal{R}_{\sigma}(b)+\gamma)\neq\emptyset\quad
\mathrm{for}\ a\in\mathcal{A}\ \mathrm{and}\ (\gamma,b)\in\Gamma.
\]
In view of Theorem~\ref{mt} and 
Lemma~\ref{t_criterion}, it is enough to show
that no other cases than $a=b$ and $\gamma=0$
can occur.

Denote the mapping $\omega\mapsto\Psi(\omega)+\gamma$ by $\Psi+\gamma$.
The subset
\[
(\Psi\times\Psi+\gamma)^{-1}\{ (X,X):
X\in \mathcal{R}_{\sigma}(a)\cap(\mathcal{R}_{\sigma}(b)+\gamma)\}
\subset\Sigma_{A}(a)\times \Sigma_{A}(b) 
\]
has measure  zero with respect to $m\times m$.
We perturb it to make Ergodic theory applicable.

%%%
%%%
%%%
\begin{lem}
Take an open set $B_{0}$ so small that 
$\overline{B}_{0}
\subset
\mathrm{int}\mathcal{R}_{\sigma}(a)\cap\mathrm{int}
(\mathcal{R}_{\sigma}(b)+\gamma)$.
Then one can find $\epsilon>0$ so that 
\[
\overline{B}_{0}\subset
\mathrm{int}\mathcal{R}_{\sigma}(a)\cap\mathrm{int}
(\mathcal{R}_{\sigma}(b)+\gamma+\zeta)\quad\mathrm{for\ all\ }\zeta\in B(0,\epsilon).
\]
\end{lem}
\begin{proof}
Write $d_{K}(X,\overline{B}_{0})=\inf_{Y\in \overline{B}_{0}} d_{K}(X,Y)$. It is standard that 
\begin{equation}\label{distance}
\vert d_{K}(X,\overline{B}_{0})-d_{K}(Y,\overline{B}_{0})\vert
\leq d_{K}(X,Y).
\end{equation}
Set
\[ 
2\epsilon=\inf_{Y\in\partial \mathcal{R}_{\sigma}(b)+\gamma}d_{K}(Y,\overline{B}_{0})>0.
\]
If $Y\in\partial \mathcal{R}_{\sigma}(b)+\gamma$ and $\zeta\in B(0,\epsilon)$,
replacing $X$ and $Y$ in (\ref{distance}) by $Y$ and $Y+\zeta$ 
yields 
$d_{K}(Y+\zeta,\overline{B}_{0})\geq\epsilon$, which completes the proof.
\end{proof}

Set $B=\Psi^{-1}B_{0}$ and define
\[
\Delta_{\epsilon}
=\bigcup_{\omega_{1}\in B}\{ \omega_{1}\}\times (\Psi+\gamma)^{-1}B(\Psi(\omega_{2}
)+\gamma,\epsilon),
\quad \Psi(\omega_{1})=\Psi(\omega_{2})+\gamma.
\]

\begin{rem}A more intuitive form is
\[
\Delta_{\epsilon}
=\bigcup_{\zeta\in B(0,\epsilon)}
\bigcup_{\omega_{1}\in B}\{ \omega_{1}\}\times (\Psi+\gamma+\zeta)^{-1}\Psi(\omega_{1}).
\]
\end{rem}
%%%
%%%
%%%
\begin{lem}\label{positive}
$\Delta_{\epsilon}$
is measurable and has a positive measure in $\Sigma_{A}(a)\times\Sigma_{A}(b)$.
\end{lem}
\begin{proof}
First we show that $\Delta_{\epsilon}$ is open in  $\Sigma_{A}(a)\times\Sigma_{A}(b)$ and
hence measurable.
Take any $(\omega_{1},\omega_{2})\in \Delta_{\epsilon}$.
This means $\omega_{2}\in (\Psi+\gamma)^{-1}B(\Psi(\omega_{1}),\epsilon)$.
Then there exists $\rho_{0}>0$
so that
\[
d_{K}(\Psi(\omega_{1}),\Psi(\omega_{2})+\gamma)=\epsilon-\rho_{0}.
\]
By the continuity of $\Psi$ (Lemma~\ref{continuous}), we can choose 
$\rho>0$ and $\rho^{\prime}>0$ so that for all $\tilde{\omega}\in
B(\omega_{1},\rho)$ and $\hat{\omega}\in B(\omega_{2},\rho^{\prime})$,
\[
d_{K}(\Psi(\tilde{\omega}),\Psi(\omega_{1}))<\frac{\rho_{0}}{2},\quad
d_{K}(\Psi(\hat{\omega})+\gamma,\Psi(\omega_{2})+\gamma)<\frac{\rho_{0}}{2}.
\]
Hence 
\[
d_{K}(\Psi(\tilde{\omega}),\Psi(\hat{\omega})+\gamma)
\]
\[
\leq
d_{K}(\Psi(\tilde{\omega}),\Psi(\omega_{1}))+d_{K}(\Psi(\omega),\Psi(\omega_{2})+\gamma)
+d_{K}(\Psi(\omega_{2})+\gamma,\Psi(\hat{\omega})+\gamma)<\epsilon,
\]
which means $B(\omega_{1},\rho)\times B(\omega_{2},\rho^{\prime})\subset\Delta_{\epsilon}$.
Hence $\Delta_{\epsilon}$ is open and Borel measurable.

Since each section $\{ \omega_{1}\}\times (\Psi+\gamma)^{-1}
B(\Psi(\omega_{2})
+\gamma,\epsilon)$ of $\Delta_{\epsilon}$ is a non-empty open set in $\{
\omega_{1}\}\times\Sigma_{A}$ and 
has a positive measure with respect to $m$,
it follows that
$m\times m(\Delta_{\epsilon})>0$ by Fubini's theorem.
\end{proof}

\section{Special Cylinder}\label{specialcylinder}

%%%
%%%
%%%
\begin{lem}[Adelic version of Garsia's lemma]\label{adelicgarsia}
Let $F$ be a polynomial of degree at most $d$ with integer coefficients 
having an upper bound
$M$ in modulus. If $F(\alpha)\neq 0$, then
\[
\prod_{v\in\mathfrak{M}}\vert F(\alpha_{v})\vert_{v}
\geq\frac{1-\alpha^{-1}}{\alpha^{d} M}.
\]
\end{lem}
\begin{proof}
When $M_{\sigma}$ is unimodular, this comes down to Garsia's lemma
(Lemma~1.51 of \cite{garsia1962arithmetic}).
Since $F(\alpha)\in\mathfrak{O}_{K}$, it is standard that
$\vert F(\alpha)\vert_{v}\leq 1$ for a finite prime $v$. 
If $\mathfrak{S}$ is a subset of finite primes, it is clear that
$\prod_{v\in\mathfrak{S}}\vert F(\alpha)\vert_{v}\leq 1$.
From the product formula, then it follows that
\[
\prod_{i=1}^{r}\vert F(\alpha_{i})\vert\prod_{i=r+1}^{r+s}\vert F(\alpha_{i})\vert^{2}
\prod_{v\in\mathfrak{M}\backslash\mathfrak{M}_{\infty}}\vert F(\alpha)\vert_{v}\geq 1.
\]
Combining this with $\vert F(\alpha)\vert\leq M\alpha^{d}\sum_{i=0}^{\infty}\alpha^{-i}$,
the proof is complete.
\end{proof}

Let $u=(u_{i})_{i\in\mathbb{Z}}$ be a fixed point of $\sigma$: $\sigma(u)=u$.
Then $\sigma(u_{-1})$ ends with $u_{-1}$, and
$\sigma(u_{0})$ starts with $u_{0}$.
In our discussion, the vertex
$(u_{0}:\emptyset)\in \mathcal{E}$ will play a role.
One feature is that it can make a transition to itself.

\begin{defn}Let $L\geq (N+1)$ be an integer.
If $c_{i}=(u_{0}:\emptyset)$ for $0\leq i\leq L$,
define the {\it special cylinder} $\mathcal{C}$ by 
\[
\mathcal{C}=\langle  c_{0}c_{1}\cdots c_{L}\rangle
=\langle \underbrace{(u_{0}:\emptyset)\cdots
(u_{0}:\emptyset)}_{L+1\ \text{terms}}\rangle.
\]
For $v\in\mathfrak{M}$, set
\[
M_{v}=\max\{\vert\langle f(p),
\vector{v}(\alpha_{v})\rangle\vert: p\in\mathrm{Pref}\}.
\]
Given an upper bound $M>0$,
the integer $L>0$ must be taken so large that
\begin{align*}
2\vert \alpha_{v}\vert^{L-n+3}\frac{M_{v}}{1-\vert\alpha_{v}\vert}
\Big(
\frac{M}{1-\alpha^{-1}}\Big)^{\frac{1}{n-1}}+\vert\alpha_{v}\vert
<1\quad &\mathrm{for}\quad v\in\mathfrak{M}_{\infty}\backslash\{ v_{1}\},\\
2\Big(\frac{1}{q_{i}}\Big)^{v_{i}(L-n+3)}+\Big(\frac{1}{q_{i}}\Big)^{v_{i}}
<1\quad &\mathrm{for}\quad v=v_{\mathfrak{p}_{i}}\quad (1\leq i\leq \kappa).
\end{align*}
\end{defn}

%%%
\begin{defn}Given an integer $d>0$ and an upper bound $M>0$, set
\[
R_{v}=
\begin{cases}
\vert\alpha_{v}
\vert^{d+n-2}\Big(\frac{1-\alpha^{-1}}{M}\Big)^{\frac{1}{d-1}} 
&\mathrm{for\ }v\in \mathfrak{M}_{\infty}\backslash\{ v_{1}\},\\
q_{i}^{-(d+n-2)\nu_{i}}& \mathrm{for\ }v=v_{\mathfrak{p}_{i}}
\in\mathfrak{M}\backslash\mathfrak{M}_{\infty}.
\end{cases}
\]
An open set of $K_{\sigma}$ 
\[
O_{d}(X)=\prod_{v\in\mathfrak{M}_{\infty}\backslash\{ v_{1}\}}
\{ Y_{v}\in K_{v}:\vert X_{v}-Y_{v}\vert<
R_{v}\}\times
\prod_{v\in\mathfrak{M}\backslash\mathfrak{M}_{\infty}}\{
Y_{v}\in K_{v}:\vert X_{v}-Y_{v}\vert_{v}<R_{v}\}
\]
is called the {\it $d$-neighborhood} of $X=(X_{v})_{v\in\mathfrak{M}}\in K_{\sigma}$.
\end{defn}

\section{Polynomial-like Behavior}\label{poly_like}

When $(\omega_{1},\omega_{2})\in\Sigma_{A}\times\Sigma_{A}$ visits $\mathcal{C}\times \mathcal{C}$ simultaneously 
at $d$, i.e.,
\[
(T\times T)^{d}(\omega_{1},\omega_{2})\in \mathcal{C}\times \mathcal{C}, 
\]
the {\it polynomial part of $\Psi(\omega_{1})-\Psi(\omega_{2})-\gamma$ at $d$}
is defined to be
\[
F(\alpha)=\sum_{k=0}^{d-1}\alpha^{k}\langle f(p_{k}),\vector{v}\rangle-\sum_{k=0}^{d-1}
\alpha^{k}\langle f(q_{k}),\vector{v}\rangle-\langle\vector{w},\vector{v}\rangle
\]
where $\omega_{1}=(a_{k+1}:p_{k})_{k\geq 0},\omega_{2}=(b_{k+1}:q_{k})_{k\geq 0}$ and
$\gamma=\Phi^{\prime}(\langle \vector{w},\vector{v}\rangle)$.
If $d=0$, then we set $F(\alpha)=-\langle\vector{w},\vector{v}\rangle$.
Similarly defined are the polynomial parts of $\Psi(\omega_{1})$ and $\Psi(
\omega_{2})+\gamma$ at
$d$ in an obvious manner.

If $\vector{a}={}^{t}(1,\alpha,\ldots,\alpha^{n-1})$, then $\vector{v}$ can be written in the 
form $B\vector{a}$ for
some $n\times n$ matrix $B$ with integer coefficients. Rewriting summands in $F(\alpha)$ as
\[
\langle f(q_{k})-f(q_{k}),\vector{v}\rangle=\langle{}^{t}B(f(q_{k})-f(q_{k})),\vector{a}\rangle,
\]
it is easy to see that an upper bound  $M$ for the coefficients of $F(\alpha)$
can be taken independently of $\omega_{1},\omega_{2}$ and $d$. Throughout the rest of this paper,
$M$ will denote this upper bound. 

%%%
%%%
%%%
\begin{thm}\label{either}
If $(\omega_{1},\omega_{2})$ visits  $\mathcal{C}\times \mathcal{C}$ simultaneously at $d$,
then either 
the polynomial part of $\Psi(\omega_{1})-\Psi(\omega_{2})-\gamma$ at $d$
vanishes  or $\Psi(\omega_{2})+\gamma\not\in O_{d+1}(\Psi(\omega_{1}))$.
\end{thm}

\begin{proof}
Set $\omega_{1}=(a_{k+1}:p_{k})_{k\geq 0},\omega_{2}=(b_{k+1}:q_{k})_{k\geq 0}$ and
$\gamma=\Phi^{\prime}(\langle \vector{w},\vector{v}\rangle)$.
Suppose that
the polynomial part of $\Psi(\omega_{1})-\Psi(\omega_{2})-\gamma$ at $d$
does not vanish. Write $X=\Psi(\omega_{1})$ and $Y=\Psi(\omega_{2})+\gamma$, and assume that
$Y\in O_{d+1}(X)$. Then
\begin{align}
&\vert X_{v}-Y_{v}\vert <\vert\alpha_{v}\vert^{d+1+n-2}\Big(
\frac{1-\alpha^{-1}}{M}\Big)^{\frac{1}{n-1}}\quad & (v\in\mathfrak{M}_{\infty}\backslash\{ v_{1}\})
\label{infty},\\
&\vert X_{v}-Y_{v}\vert_{v} \leq\frac{1}{q_{i}^{\nu_{i}(d+1+n-2)}}\quad &(
v=v_{\mathfrak{p}_{i}}
\in\mathfrak{M}\backslash\mathfrak{M}_{\infty})\label{finite}.
\end{align}
Let $F^{X}(\alpha)$ and $F^{Y}(\alpha)$
be the polynomial parts of $X$ and $Y$ at $d$ respectively.
Observe that $F(\alpha)=F^{X}(\alpha)-F^{Y}(\alpha)$
is a polynomial of degree $d+n-2$ with integer coefficients
having the upper bound $M$.
If $\Phi(F^{X}(\alpha))=(F^{X}_{v})_{v}$, then
\[
X_{v}-F_{v}^{X}=\alpha_{v}^{d+L+1}\sum_{k=d+L+1}^{\infty}\alpha_{v}^{k-d-L-1}\langle f(p_{k}),
\vector{v}(\alpha_{v})\rangle.
\]
For $v\in\mathfrak{M}_{\infty}\backslash\{ v_{1}\}$, we have
%%%
\begin{equation}\label{infty2}
\vert X_{v}-F_{v}^{X}\vert\leq\vert \alpha_{v}\vert^{d+L+1}\frac{M_{v}}{1-\vert\alpha_{v}\vert}.
\end{equation}
For $v\in\mathfrak{M}\backslash\mathfrak{M}_{\infty}$,
observe that $\xi=\sum_{k=d+L+1}^{\infty}\alpha_{v}^{k-d-L-1}
\langle f(p_{k}),\vector{v}(\alpha_{v})\rangle$ is the limit of
a Cauchy sequence of $\mathfrak{O}_{K}$ in $K_{v}$, and hence $\vert\xi\vert_{v}\leq 1$.
So
%%%
\begin{equation}\label{finite2}
\vert X_{v}-F_{v}^{X}\vert_{v}\leq\Big(\frac{1}{q_{i}}\Big)^{v_{\mathfrak{p}_{i}}
(\alpha_{v}^{d+L+1}\xi)}\leq
\Big(\frac{1}{q_{i}}\Big)^{v_{i}(d+L+1)}.
\end{equation}
Similar estimates for $Y-F^{Y}(\alpha)$ hold. 

The valuation inequality for $v=v_{\mathfrak{p}_{i}}\in
\mathfrak{M}\backslash\mathfrak{M}_{\infty}$
\[
\vert F_{v}^{X}-F_{v}^{Y}\vert_{v}\leq\vert F_{v}^{X}-X_{v}\vert_{v}
+\vert X_{v}-Y_{v}\vert_{v}+\vert Y_{v}-F_{v}^{Y}\vert_{v}
\]
together with (\ref{finite}) and (\ref{finite2}) implies
\[
\vert F_{v}^{X}-F_{v}^{Y}\vert_{v}\leq\Big(\frac{1}{q_{i}}\Big)^{v_{i}(d+n-2)}
\Big[2\Big(\frac{1}{q_{i}}\Big)^{v_{i}(L-n+3)}+\Big(\frac{1}{q_{i}}\Big)^{v_{i}}\Big].
\]
The similar inequality for absolute values together with (\ref{infty}) and (\ref{infty2}) yields
\[
\vert F_{v}^{X}-F_{v}^{Y}\vert\leq 
\vert\alpha_{v}\vert^{d+n-2}\Big(
\frac{1-\alpha^{-1}}{M}\Big)^{\frac{1}{n-1}}
\Big[2\vert \alpha_{v}\vert^{L-n+3}\frac{M_{v}}{1-\vert\alpha_{v}\vert}
\Big(
\frac{M}{1-\alpha^{-1}}\Big)^{\frac{1}{n-1}}+\vert\alpha_{v}\vert\Big].
\]
Using (\ref{product_alpha}) and the definition of $L$, we obtain 
\[
\prod_{v\in\mathfrak{M}}\vert F_{v}^{X}-F_{v}^{Y}\vert_{v}
< \frac{1-\alpha^{-1}}{\alpha^{d+n-2} M},
\]
which contradicts Lemma~\ref{adelicgarsia}. This completes the proof.
\end{proof}

\section{Ergodic part}\label{ergodic}

Define a measurable function $\tau_{2}:\Sigma_{A}\times\Sigma_{A}\to[0,\infty]$ by
\[
\tau_{2}(\omega_{1},\omega_{2})=\inf\ \{ k\geq 0: (T\times T)^{k}(
\omega_{1},\omega_{2})\in \mathcal{C}\times\mathcal{C}\}.
\]
If there is no such $k$, put $\tau_{2}(\omega_{1},\omega_{2})=\infty$.
Notice that our concern is the {\it first entry time} to $\mathcal{C}\times\mathcal{C}$
rather than the first return time, which is defined only in $\mathcal{C}\times\mathcal{C}$.
We write 
\[
\{ \tau_{2}=k\}:=\{ (\omega_{1},\omega_{2}): \tau_{2}(\omega_{1},\omega_{2})=k\}.
\] 
The $\{ \tau_{2}=k\}$ are pairwise disjoint and
the measurability of $\tau_{2}$ follows from
\[
\{ \tau_{2}=0\}=\mathcal{C}\times\mathcal{C},\quad
\{ \tau_{2}=k\}=(T\times T)^{-k}\mathcal{C}\times\mathcal{C}-\bigcup_{i=0}^{k-1}\{ \tau_{2}=i\}
\quad (k\geq 1)
\]
and $\{ \tau_{2}=\infty\}=\Sigma_{A}\times\Sigma_{A}-\bigcup_{k-0}^{\infty}\{ \tau_{2}=k\}$.
Since almost every $(\omega_{1},\omega_{2})\in\Sigma_{A}\times\Sigma_{A}$ visits infinitely often
$\mathcal{C}\times\mathcal{C}$, it follows that $m\times m(\{ \tau_{2}=\infty\})=0$. Hence
$\Sigma_{A}\times\Sigma_{A}$
admits a measurable partition
\begin{equation}\label{ergodic_partition}
\Sigma_{A}\times\Sigma_{A}=\bigsqcup_{k\geq 0}\{ (\omega_{1},\omega_{2}): \tau_{2}(\omega_{1},\omega_{2}
)=k\}\quad(\mathrm{up\ to\ sets\ of\
 measure}\ 0).
\end{equation}

%%%
When a property $P$
holds almost everywhere in $\Delta_{\epsilon}$,
$m$-almost every fiber of $\Delta_{\epsilon}$ 
shares the same property. More precisely, let $\mathcal{N}(P)$ be
the null set with respect to $m\times m$ for which $P$ does not hold.
Then
\[
m\times m(\Delta_{\epsilon})=m\times m\Big(\Delta_{\epsilon}\backslash\mathcal{N}(P)
).
\]
Writing $\tilde{B}(\omega_{1})=(\Psi+\gamma)^{-1}B(\Psi(\omega_{1}),\epsilon)$ and
applying Fubini's theorem to both sides, we obtain
\[
\int_{B}m(\tilde{B}(\omega_{1})
)dm(\omega_{1})=\int_{B}m\Big(
\{ \omega_{2}\in \tilde{B}(\omega_{1}): (\omega_{1},\omega_{2})\not\in \mathcal{N}(P)
\}\Big)dm(\omega_{1}).
\]
As the integrand on the righthand side is less than the one on the left, we conclude
that for $m$-almost every $\omega_{1}\in B$,
\[
m((\Psi+\gamma)^{-1}B(\Psi(\omega_{1}),\epsilon))=
m\Big(\{ \omega_{2}\in (\Psi+\gamma)^{-1}B(\Psi(\omega_{1}),\epsilon): 
(\omega_{1},\omega_{2})\not\in \mathcal{N}(P)\}\Big).
\]

%%%
\begin{lem}\label{partition}
The fiber at almost every $\omega_{1}\in B$
\[
\Delta_{\epsilon}(\omega_{1}):=
\{ \omega_{1}\}\times (\Psi+\gamma)^{-1}B(\Psi(\omega_{1}),\epsilon)
\]
has a decomposition
\[
\Delta_{\epsilon}(\omega_{1})
=\{ \omega_{1}\}\times \bigsqcup_{k\geq 0}\{ \omega_{2}\in (\Psi+\gamma)^{-1}B(\Psi(\omega_{1}),
\epsilon):
 \tau_{2}(\omega_{1},\omega_{2})=k\}
\]
up to sets of measure $0$.
\end{lem}
\begin{proof}
Since $m\times m$ is ergodic, it follows by (\ref{ergodic_partition}) that
\[
m\times m(\Delta_{\epsilon})=m\times m\Big(\Delta_{\epsilon}\cap
\bigsqcup_{k\geq 0}\{ (\omega_{1},\omega_{2})\in\Sigma_{A}\times\Sigma_{A}
: \tau_{2}(\omega_{1},\omega_{2}
)=k\}\Big).
\]
Thus for $m$-almost every $\omega_{1}\in B$
\[
m((\Psi+\gamma)^{-1}B(\Psi(\omega_{1}),\epsilon))=
m\Big(\bigsqcup_{k\geq 0}\{ \omega_{2}\in (\Psi+\gamma)^{-1}B(\Psi(\omega_{1}),\epsilon): 
\tau_{2}(\omega_{1},\omega_{2})=k\}\Big).
\]
\end{proof}

\section{First Entry To $\mathcal{C}$ Along A Time Series}

Let 
\[
\mathcal{C}_{0}=(\Sigma_{A}\backslash\mathcal{C})\cap T^{-1}(\Sigma_{A}\backslash\mathcal{C})\cap
\cdots\cap T^{-(N+1)}(\Sigma_{A}\backslash\mathcal{C})
\]
be the set of all those points of $\Sigma_{A}$ which 
do not enter $\mathcal{C}$ by $T^{i}$ for $0\leq i\leq N+1$.

%%%
%%%
\begin{rem}
Clearly
$m(\mathcal{C}_{0})>0$ beccause
$\langle a_{0}\cdots a_{L}\rangle\subset\mathcal{C}_{0}$
if $a_{L}\neq(u_{0}:\emptyset)$.
\end{rem}

From this point forward, fix such  a fiber $\Delta_{\epsilon}(x)$ at $x\in B$
that the decomposition of Lemma~\ref{partition} holds for
$\Delta_{\epsilon}(x)$ and, moreover, that 
%%%
%%%
\begin{equation}\label{C_0}
\mathrm{a.e.}\ (x,\omega_{2})\in\Delta_{\epsilon}(x)\
\mathrm{visits}\ \mathcal{C}\times\mathcal{C}\  and\
\mathcal{C}_{0}\times\Sigma_{A}\
\mathrm{infinitely\ often.}
\end{equation}

%%%
%%%
\begin{lem}\label{N+}
There exists  $N_{+}>0$ so that 
$O_{N_{0}}(\Psi(x))\subset B(\Psi(x),\epsilon)$ for all $N_{0}\geq N_{+}$
and that 
\begin{equation}\label{inside}
\Psi(\langle \omega_{0}\cdots \omega_{N_{0}-(N+1)} \rangle)+\gamma
\subset B(\Psi(x),\epsilon)
\end{equation}
whenever $\omega=(\omega_{i})_{i\geq 0}\in(\Psi+\gamma)^{-1}O_{N_{0}}(\Psi(x))$.
\end{lem}
\begin{proof}
If $N_{+}$ is sufficiently large, clearly
$O_{N_{0}}(\Psi(x))\subset B(\Psi(x),\epsilon)$ for $N_{0}\geq N_{+}$.
For $\omega\in(\Psi+\gamma)^{-1}O_{N_{0}}(\Psi(x))$, set
$Y=\Psi(\omega)+\gamma$. Then $Y\in O_{N_{0}}(\Psi(x))$.
By the triangle inequality
\[
d_{K}(\Psi(x),Z)\leq d_{K}(\Psi(x),Y)+d_{K}(Y,Z)\quad 
\mathrm{for}\ Z\in \Psi(\langle \omega_{0}\cdots \omega_{N_{0}-(N+1)} \rangle)+\gamma,
\]
(\ref{inside}) follows immediately.
\end{proof}

Now select $N_{0}$ so that $N_{0}-(N+1)\geq N_{+}$ and that
$T^{N_{0}-(N+1)}x\in \mathcal{C}_{0}$ as in Lemma~\ref{N+}.
Then take $y=(y_{i})_{i\geq 0}\in(\Psi+\gamma)^{-1}O_{N_{0}}(\Psi(x))$
so that $(x,y)$ satisfies (\ref{C_0}).

Since $A^{N+1}>0$, one can choose $z=(z_{i})_{i\geq 0}\in\langle y_{0}\cdots y_{N_{0}-(N+1)}\rangle$
with $z_{N_{0}}=(u_{0}:\emptyset)$, so that
\[
\Psi(\langle z_{0}\cdots z_{N_{0}} \rangle)+\gamma
\subset \Psi(\langle y_{0}\cdots y_{N_{0}-(N+1)} \rangle)+\gamma\subset
B(\Psi(x),\epsilon)
\]
by (\ref{inside}). It follows that
%%%
\begin{equation}\label{fibre_inclusion}
\{ x\}\times\langle z_{0}\cdots z_{N_{0}} \rangle\subset\Delta_{\epsilon}(x). 
\end{equation}
Observe that by construction
%%%
%%%
\begin{equation}\label{x_construction}
T^{i}x\not\in\mathcal{C}\quad \mathrm{for}\ N_{0}-(N+1)\leq i\leq N_{0}.
\end{equation}

Let $\{ N_{k}\}_{k=1}^{\infty}$ be the set of the entry times of $x$ 
to $\mathcal{C}$ after $N_{0}$:
\[
T^{N_{k}}x\in\mathcal{C},\quad N_{k}>N_{0}\quad(k\geq 1).
\]
We say that {\it
the first entry of $\omega$
to $\mathcal{C}$ along the time series
$\{ N_{0},N_{1},\ldots \}$ is $N_{k}$} if
\begin{equation}\label{timeseries}
\begin{split}
T^{N_{0}}\omega\in\mathcal{C}\ &\qquad\mathrm{for}\ k=0, \\
T^{N_{i}}\omega\not\in\mathcal{C}\ (0\leq i\leq k-1),\ T^{N_{k}}\omega\in\mathcal{C}
\ &\qquad\mathrm{for}\ k\geq 1.
\end{split}
\end{equation}
Set
\[ 
\begin{split}
s_{j}=& m\Big(\{\omega\in \langle z_{0}\cdots z_{N_{0}} \rangle:
T^{N_{0}}\omega\in\mathcal{C}\}\Big) \\
&+
\sum_{k=1}^{j}m\Big(\{\omega\in \langle z_{0}\cdots z_{N_{0}} \rangle:
T^{N_{i}}\omega\not\in\mathcal{C}\ (0\leq i\leq k-1),\ T^{N_{k}}\omega\in\mathcal{C}\}\Big).
\end{split}
\]
For $k\geq 0$, let $b(N_{k})$ be the number of distinct subcylinders  of 
$\langle z_{0}\cdots z_{N_{0}} \rangle$  of the form
\[
\langle z_{0}\cdots z_{N_{0}}\omega_{N_{0}+1}\cdots\omega_{N_{k}-1}  \underbrace{I\cdots I
}_{L+1\ \text{terms}} 
\rangle
\]
whose elements  satisfy (\ref{timeseries}) with $I=(u_{0}:\emptyset)$. 
In case of $k=0$, the word $\omega_{N_{0}+1}\cdots\omega_{N_{k}-1}I$ is void and
$b(N_{0})=1$ by construction.
Observe that such subcylinders  give the same measure
\[ 
m(\langle z_{0}\cdots z_{N_{0}}\omega_{N_{0}+1}\cdots\omega_{N_{k}-1}  \underbrace{I\cdots I
}_{L+1\ \text{terms}} 
\rangle)
=
\frac{1}{\alpha^{N_{k}-N_{0}+L}}
m(\langle
z_{0}\cdots z_{N_{0}} \rangle)
\]
irrespective of admissible paths $\omega_{N_{0}+1}\cdots\omega_{N_{k}-1}$.
So
\[
s_{j}
=b(N_{0})\cdot
\frac{1}{\alpha^{L}}m(\langle z_{0}\cdots z_{N_{0}} \rangle)
+\sum_{k=1}^{j}b(N_{k})\cdot
\frac{1}{\alpha^{N_{k}-N_{0}+L}}m(\langle
z_{0}\cdots z_{N_{0}} \rangle).
\]
Notice that $s_{j}$ is a sum of the measures of disjoint subsets of 
$m(\langle z_{0}\cdots z_{N_{0}} \rangle)$, so that it obviously converges as
$j$ goes to infinity.
%%%
%%%
%%%
\begin{proposition}\label{s_infty}
If we write $s_{\infty}=\lim_{j\to\infty}s_{j}$, then
\[
s_{\infty}=m(\langle z_{0}\cdots z_{N_{0}} \rangle).
\]
\end{proposition}
\begin{proof}
Since $\{N_{k}\}_{k}$ is strictly monotone,
for each $k\geq L+2$ there exists $l=l(k),0\leq l\leq L+1$, 
so that
\[
N_{k}-N_{k-1}\leq L,\ldots,
N_{k}-N_{k-(l-1)}\leq L,N_{k}-N_{k-l}>L.
\]
Recall that we have taken $z_{N_{0}}=I$ and that 
\[
a^{(N_{k}-N_{0})}_{II}=\#\Big\{\ 
z_{N_{0}}\omega_{N_{0}+1}\cdots\omega_{N_{k}-1}  \underbrace{I\cdots I
}_{L+1\ \text{terms}} :
\mathrm{admissible \ paths\ of\ length}\ N_{k}-N_{0}+1
\Big\}.
\]
Then
%%%
\begin{equation}\label{b}
b(N_{k})=a^{(N_{k}-N_{0})}_{II}-b(N_{k-1})-\cdots-b(N_{k-(l-1)})-\sum_{i=0}^{k-l}
a_{II}^{(N_{k}-(N_{i}+L))}b(N_{i}).
\end{equation}
Put $E=m(\langle z_{0}\cdots z_{N_{0}}\rangle)$.
Multiplying $E\alpha^{-(N_{k}-N_{0}+L)}$ on both sides of (\ref{b}), we get
%%%
%%%
%%%
\begin{equation}\label{s_infinity}
\begin{split}
\frac{b(N_{k})}{\alpha^{N_{k}-N_{0}+L}}E
=\frac{a^{(N_{k}-N_{0})}_{II}}{\alpha^{N_{k}-N_{0}+L}}E
-&\frac{\alpha^{N_{k-1}}}{\alpha^{N_{k}}}\cdot
\frac{b(N_{k-1})}{\alpha^{N_{k-1}-N_{0}+L}}E \\
-\cdots-\frac{\alpha^{N_{k-(l-1)}}}{\alpha^{N_{k}}}\cdot
\frac{b(N_{k-(l-1)})}{\alpha^{N_{k-(l-1)}-N_{0}+L}}E
-&\sum_{i=0}^{k-l}\frac{
a_{II}^{(N_{k}-N_{i}-L)}}{\alpha^{N_{k}-N_{i}}}
\frac{b(N_{i})}{\alpha^{N_{i}-N_{0}+L}}E.
\end{split}
\end{equation}

For $k\geq L+2$ and $0\leq i\leq k-l$, denote by $E(i,k)$ the modulus of
\[
\sum_{l=2}^{r}\alpha_{l}^{N_{k}-(N_{i}+L)}
\frac{[\vector{u}_{l}]_{I}[\vector{v}_{l}]_{J}}{
\langle[\vector{u}_{l}],[\vector{v}_{l}]\rangle}+\sum_{l=r+1}^{r+s}\alpha_{l}^{N_{k}-(N_{i}+L)}
\frac{[\vector{u}_{l}]_{I}[\vector{v}_{l}]_{J}}{
\langle[\overline{\vector{u}_{l}}],[\vector{v}_{l}]\rangle}
+\overline{\alpha_{l}}^{N_{k}-(N_{i}+L)}
\frac{\overline{[\vector{u}_{l}]_{I}}\overline{[\vector{v}_{l}]_{J}}}{
\langle[\vector{u}_{l}],[\overline{\vector{v}_{l}}]\rangle}.
\]
Then $E(i,k)$ is bounded by a constant $\overline{E}$ irrespective of $k$ and $i$.
Let $\epsilon>0$ and find an integer $K_{0}\geq L+1$ so that $\alpha^{-K_{0}}\overline{E}
<\epsilon s_{\infty}^{-1}$.
Since $N_{k}-N_{k-K_{0}}\geq K_{0}$, it follows by Theorem~\ref{asympt}
that
\[
\Big\vert\frac{
a_{II}^{(N_{k}-(N_{i}+L))}}{\alpha^{N_{k}-N_{i}}}-
\frac{[\vector{u}_{1}]_{I}[
\vector{v}_{1}]_{I}}{\alpha^{L}
\langle[\vector{u}_{1}],[\vector{v}_{1}]\rangle}\Big\vert=\frac{1}{\alpha^{N_{k}-N_{i}}}E(i,k)
<\frac{\epsilon}{s_{\infty}}
\]
for $0\leq i\leq k-K_{0}$ and $k\geq\max\{ L+2,K_{0}\}$. Thus
\[
\Big\vert \sum_{i=0}^{k-K_{0}}\frac{
a_{II}^{(N_{k}-(N_{i}+L))}}{\alpha^{N_{k}-N_{i}}}
\frac{b(N_{i})}{\alpha^{N_{i}-N_{0}+L}}E-\frac{[\vector{u}_{1}]_{I}[
\vector{v}_{1}]_{I}}{\alpha^{L}
\langle[\vector{u}_{1}],[\vector{v}_{1}]\rangle}\sum_{i=0}^{k-K_{0}}
\frac{b(N_{i})}{\alpha^{N_{i}-N_{0}+L}}E\Big\vert<\epsilon.
\]
Applying this to (\ref{s_infinity}), we obtain
\begin{equation}\label{epsilon-error}
\Big\vert \Sigma-
\frac{a^{(N_{k}-N_{0})}_{II}}{\alpha^{N_{k}-N_{0}+L}}E+
\frac{[\vector{u}_{1}]_{I}[
\vector{v}_{1}]_{I}}{\alpha^{L}
\langle[\vector{u}_{1}],[\vector{v}_{1}]\rangle}\sum_{i=0}^{k-K_{0}}
\frac{b(N_{i})}{\alpha^{N_{i}-N_{0}+L}}E\Big\vert<\epsilon.
\end{equation}
where
\[
\Sigma=\sum_{i=0}^{l-1}\frac{\alpha^{N_{k-i}}}{\alpha^{N_{k}}}
\frac{b(N_{k-i})}{\alpha^{N_{k-i}-N_{0}+L}}E+
\sum_{i=k-K_{0}+1}^{k-l}\frac{a_{II}^{(N_{k}-(N_{i}+L))}}{\alpha^{N_{k}-N_{i}}}
\frac{b(N_{i})}{\alpha^{N_{i}-N_{0}+L}}E.
\]
Since each $b(N_{i})\alpha^{-(N_{i}-N_{0}+L)}E$ 
for $k-K_{0}+1\leq i\leq k$ goes to zero
as $k\to\infty$, so does $\Sigma$.
Letting $k\to\infty$ in (\ref{epsilon-error}), it follows that

\[
\vert -E
+s_{\infty}\vert
\leq \frac{\alpha^{L}
\langle[\vector{u}_{1}],[\vector{v}_{1}]\rangle}{[\vector{u}_{1}]_{I}[
\vector{v}_{1}]_{I}}\epsilon.
\]
Since $\epsilon>0$ is arbitrary, this proves the proposition.
\end{proof}

\section{Non-zero Polynomial Parts Distort The Distribution Of $\tau_{2}$}\label{distorsion}

%%%
\begin{proposition}\label{distort}
There exists $d\geq 0$ so that
the polynomial part of $\Psi(x)-\Psi(y)-\gamma$
at $d$ vanishes.
\end{proposition}
\begin{proof}
First notice that $(x,y)$ visits $\mathcal{C}\times\mathcal{C}$ infinitely often by construction.
Assume that the polynomial part of $\Psi(x)-\Psi(y)-\gamma$
at any possible time never vanishes.
Then $\tau_{2}(x,y)\geq N_{0}$:
indeed,
if $\tau_{2}(x,y)< N_{0}$, Theorem~\ref{either} and the assumption
imply that
$\Psi(y)+\gamma\not\in O_{\tau_{2}(x,y)+1}(\Psi(x))$,
which contradicts the fact that $\Psi(y)+\gamma\in O_{N_{0}}(\Psi(x))$.

The condition $\tau_{2}(x,y)\geq N_{0}$ implies that
\[
(T\times T)^{i}(x,\omega)\not\in\mathcal{C}\times \mathcal{C}\quad
\mathrm{for}\ 0\leq i\leq N_{0}-(N+1),
\]
for any $\omega\in \langle y_{0}\cdots y_{N_{0}-(N+1)} \rangle$.
Moreover, by the choice of $x$ and (\ref{x_construction}),
\[
(T\times T)^{i}(x,\omega)\not\in\mathcal{C}\times \mathcal{C}\quad
\mathrm{for}\ 0\leq i\leq N_{0}
\]
for any $\omega\in\langle z_{0}\cdots z_{N_{0}}\rangle$, which means
%%%
\begin{equation}\label{zero_nuki}
\tau_{2}(x,\omega)\in \{ N_{k}\}_{k=1}^{\infty}\quad\mathrm{and}\quad
m\Big(\{ \omega\in \langle z_{0}\cdots z_{N_{0}}\rangle: \tau_{2}(x,\omega)<N_{1} \}\Big)=0.
\end{equation}
For $\omega\in\langle z_{0}\cdots z_{N_{0}}\rangle$,it follows that $\tau_{2}(x,\omega)=N_{k}$ 
if and only if 
the first entry of $\omega$ to $\mathcal{C}$ along the time series
$\{ N_{0},N_{1},\ldots \}$ is $N_{k}$.
Thus
\begin{align*}
s_{j}-m\Big(\{\omega\in \langle z_{0}\cdots z_{N_{0}} \rangle:
T^{N_{0}}\omega\in\mathcal{C}\}\Big)
 &=\sum_{k=1}^{j}m\Big(\{ \omega\in \langle z_{0}\cdots z_{N_{0}}\rangle:
 \tau_{2}(x,\omega)=N_{k} \}\Big)\\
 &=
\sum_{k=0}^{N_{j}}m\Big(\{ \omega\in \langle z_{0}\cdots z_{N_{0}}\rangle:
 \tau_{2}(x,\omega)=k \}\Big)
\end{align*}
by (\ref{zero_nuki}).
Lemma~\ref{partition} and (\ref{fibre_inclusion})
imply that as $j\to\infty$ the righthand side converges to
$m(\langle z_{0}\cdots z_{N_{0}}\rangle)$.
This contradicts Proposition~\ref{s_infty}, which concludes the theorem.
\end{proof}

By Proposition~\ref{distort}, there exists $d\geq 0$ so that
\begin{equation}\label{poly_eq}
\sum_{k=0}^{d-1}\alpha^{k}\langle f(p_{k}),\vector{v}\rangle-\sum_{k=0}^{d-1}
\alpha^{k}\langle f(q_{k}),\vector{v}\rangle-\langle\vector{w},\vector{v}\rangle=0
\end{equation}
where
\[
x=((a_{k+1}:p_{k}a_{k}s_{k}))_{k\geq 0},\quad y=((b_{k+1}:q_{k}b_{k}t_{k}))_{k\geq 0}\quad \mathrm{and}\quad
\gamma=\Phi^{\prime}(\langle\vector{w},\vector{v}\rangle).
\]
We stress that $a_{d}=b_{d}=u_{0}$ because $(T\times T)^{d}(x,y)\in\mathcal{C}\times
\mathcal{C}$ and because
\[
(a_{d+1}:p_{d}a_{d}s_{d})=(u_{0}:\emptyset u_{0}s)=(b_{d+1}:q_{d}b_{d}t_{d})
\]
where $\sigma(u_{0})=u_{0}s$.
In view of the formula 
$\alpha^{k}\langle f(p_{k}),\vector{v}\rangle=\langle f(\sigma^{k}(p_{k})),\vector{v}\rangle$,
the equation (\ref{poly_eq}) reduces to 
%%%
\begin{equation}\label{houteishiki}
f(\sigma^{d-1}(p_{d-1})\cdots p_{0})=f(\sigma^{d-1}(q_{d-1})\cdots q_{0})+\vector{w}
\end{equation}
by Lemma~\ref{latticeII} or $\mathbb{Q}$-linear independence of $\vector{v}$.
It immediately follows from (\ref{houteishiki})
that $\vector{w}$ must be in $\mathbb{Z}^{n}$.

Notice that if $\mathrm{Pref}(\sigma^{k},a)$ denotes the set of prefixes
for $\sigma^{k}(a)$,
either $p_{1}\preceq p_{2}$ or $p_{1}\succeq p_{2}$ always holds
for $p_{1},p_{2}\in\mathrm{Pref}(\sigma^{k},a)$. 
As both $\sigma^{d-1}(p_{d-1})\cdots p_{0}$ and $\sigma^{d-1}(q_{d-1})\cdots q_{0}$
are prefixes of $\sigma^{d}(u_{0})$, either of the following cases happens:
\begin{itemize}
\item[I.]$\sigma^{d-1}(p_{d-1})\cdots p_{0}=\sigma^{d-1}(q_{d-1})\cdots q_{0}$,
\item[II.]$\sigma^{d-1}(p_{d-1})\cdots p_{0}\prec\sigma^{d-1}(q_{d-1})\cdots q_{0}$,
\item[III.]$
\sigma^{d-1}(p_{d-1})\cdots p_{0}\succ\sigma^{d-1}(q_{d-1})\cdots q_{0}$.
\end{itemize}

Case I. Observe that $a_{d}=b_{d}=u_{0}$.
By Lemma~1.3 of \cite{dumont1989systemes}, it follows that
$(a_{i+1}:p_{i})=(b_{i+1}:q_{i})$
for $0\leq i\leq d-1$.
Hence $a=a_{0}=b_{0}=b$ and $\vector{w}=0$.

Case II. Denote the $i$th coordinate of $\vector{w}$ by $(\vector{w})_{i}$. By (\ref{houteishiki}),
$\vector{w}$ must be non-positive ($(\vector{w})_{i}\leq 0$ for all $i$) and
non-zero
to cancel the redundancy.
However, this contradicts the condition $\langle\vector{w},\vector{v}\rangle\geq 0$
 (recall that $\vector{v}>0$ by Perron-Frobenius theorem). So this case never happens.

Case III. There exists  a (possibly empty) suffix  $s$ so that
\begin{equation}\label{equal}
\sigma^{d-1}(p_{d-1})\cdots p_{0}=\underbrace{\sigma^{d-1}(q_{d-1})\cdots q_{0}}
_{\textrm{prefix of the left-hand side}}bs.
\end{equation}
Then $\vector{w}$ must compensate $f(bs)$.
The existence of a negative coordinate $(\vector{w})_{i}<0$ makes 
(\ref{houteishiki}) fail to hold because 
the $i$th coordinates of both sides of (\ref{houteishiki}) would not coincide.
If $\vector{w}$ is non-negative ($(\vector{w})_{j}\geq 0$ for all $j$),
then $(\vector{w})_{b}=0$
in order to fulfill the condition
$\langle\vector{w}-\vector{e}_{b},\vector{v}\rangle< 0$. This means that
$f(b)$ is not supplied by $\vector{w}$, which 
contradicts (\ref{equal}).
Hence this case never happens as well.

The preceding argument establishes the following.
%%%
%%%
%%%
\begin{thm}\label{trulytiling}
Let $\sigma$ be an irreducible Pisot substitution. Then
$\mathcal{T}$ is a tiling of $K_{\sigma}\ (d_{\mathrm{cov}}=1)$.
\end{thm}

\section{Consistency Between Classical and Modern Methods}\label{modern}

This section will elaborate on equivalent conditions between tiling and pure discrete spectrum,
which are summarized in \cite{sing2006pisot}.
See also \cite{ito2006atomic} and \cite{barge2006geometric} for unimodular cases
($\mathrm{det}M_{\sigma}=\pm 1$).

The {\it expansive matrix function system}
$\Theta=(\Theta_{ab})_{a,b\in\mathcal{A}}$ on $\mathbb{R}^{n}$ (\S 5.4 of \cite{sing2006pisot})
is defined by
%%%
%%%
\begin{equation}\label{theta_ab}
\Theta_{ab}=\bigcup_{b\xrightarrow[]{p}a}\{ t_{p}\circ g_{0}\}
\quad\mathrm{for}\  a,b\in\mathcal{A}
\end{equation}
where $g_{0}(\xi)=\alpha\xi$ and $t_{p}(\xi)=\xi+\langle f(p),\vector{v}\rangle$ are mappings
from $\mathbb{R}$ to itself. 
If $\mathcal{H}(\mathbb{R})$ denotes the space of non-empty compact subsets of $\mathbb{R}$,
then the action of 
$\Theta$ on $\mathcal{H}(\mathbb{R})^{n}$ is given by
\[
\Theta(\underline{\Lambda})=\Big(\bigcup_{b\in\mathcal{A}}
\bigcup_{g\in\Theta_{ab}}g(\Lambda_{b})
  \Big)_{a\in\mathcal{A}}
\quad
(\underline{\Lambda}=(\Lambda_{a})_{a\in\mathcal{A}}\in\mathcal{H}(
\mathbb{R}^{n})).
\]
Then $\underline{\Lambda}$ is called a {\it primitive substitution multi-component Delone set} if
$\Theta(\underline{\Lambda})=\underline{\Lambda}$.

The {\it adjoint matrix function system} $\Theta^{\#}$ can be defined by
\[
\Theta^{\#}_{ab}=\bigcup_{a\xrightarrow[]{p}b}\{ g_{0}^{-1}\circ t_{p} \}
\quad\mathrm{for}\  a,b\in\mathcal{A}.
\]
If $A_{a}=[0,\langle \vector{e}_{a},\vector{v}\rangle]$ for $a\in\mathcal{A}$
(the {\it natural intervals}), 
the vector $\underline{A}=(A_{a})_{a\in\mathcal{A}}$
gives a unique attractor  for
$\Theta^{\#}$: $\Theta^{\#}(\underline{A})=\underline{A}$.

As described in Remark~5.75 and (6.1) of \cite{sing2006pisot},
the substitution $\sigma$ gives rise to a {\it tile substitution}
%%%
%%%
\begin{equation}\label{tilesubst}
A_{a}+x\mapsto\{ A_{b}+\langle f(p),\vector{v}\rangle+\alpha x: b\in\mathcal{A},
t_{p}\circ g_{0}\in\Theta_{ba}\}
\end{equation}
with the `inflation and subdivision' rule 
\[
g_{0}(A_{a})=\bigcup_{b\in\mathcal{A}}\bigcup_{t_{p}\circ g_{0}\in\Theta_{ba}}A_{b}+\langle
f(p),\vector{v}\rangle,
\quad 
(\alpha\Theta^{\#}(\underline{A})=\alpha\underline{A}).
\]
Just as $u_{-1}.u_{0}$ play a role in the fixed point $\sigma(u)=u$, 
iterations of $A_{u_{-1}}.A_{u_{0}}$ in (\ref{tilesubst})
give
a tiling of $\mathbb{R}$ (Definition~5.82 and 
observations after Definition~6.6 of \cite{sing2006pisot}).

For $a\in\mathcal{A}$, let $\Lambda_{a}$ be the set of left endpoints of $A_{a}$ in the tiling.
If $\underline{\Lambda}=(\Lambda_{a})_{a\in\mathcal{A}}$,
it is a primitive substitution multi-component Delone set and is
also obtained by
\[
\underline{\Lambda}=\bigcup_{k\geq 0}\Theta^{k}
(\emptyset,\ldots,\emptyset,
\{ -\langle f(u_{-1}),\vector{v}\rangle \},\emptyset,\ldots,\emptyset)\cup
(\emptyset,\ldots,\emptyset,
\{ 0\},\emptyset,\ldots,\emptyset)
\]
where $\{ -\langle f(u_{-1}),\vector{v}\rangle \}$ and $\{ 0\}$ are placed at 
$u_{-1}$ and $u_{0}$ respectively.

We say that
$\underline{\Lambda}$ has {\it finite local complexity} if for every compact subset $W\subset
\mathbb{R}$,
there exists a finite set $Y\subset\bigcup_{a\in\mathcal{A}}\Lambda_{a}$ such that
\[
\forall \xi\in \bigcup_{a\in\mathcal{A}}\Lambda_{a}
\quad\exists\eta\in Y\quad (W\cap(\Lambda_{a}-\xi))_{a\in\mathcal{A}}
=(W\cap(\Lambda_{a}-\eta))_{a\in\mathcal{A}}.
\]
%%%
%%%
\begin{thm}[Definition~5.82 and Corollary~6.41 of \cite{sing2006pisot}]
Let $\sigma$ be an irreducible Pisot substitution.
Then $\underline{\Lambda}$ has finite local complexity and is representable
(i.e. $\underline{\Lambda}+\underline{A}$ is a tiling).
\end{thm}

A pair $T_{a}=(A_{a}, a)$ is a {\it tile} of $\mathbb{R}$
and define $t+T_{a}=(t+A_{a}, a)$ for $t\in\mathbb{R}$. Then
\[
\mathcal{T}_{\Lambda}=\{ \xi_{a}+T_{a}: \xi_{a}\in\Lambda_{a}\ 
(a\in\mathcal{A})\}
\]
is a tiling of $\mathbb{R}$ with finite local complexity in the sense of \cite{lee2003consequences}.
Set
\[
\mathbb{X}(\mathcal{T}_{\Lambda})=\overline{\{ -t+\mathcal{T}_{\Lambda}:t\in\mathbb{R}\}}
\]
(the orbit closure of a point $\mathcal{T}_{\Lambda}$ by $\mathbb{R}$ action with respect to
the {\it tiling metric}). 
Since $\mathcal{T}_{\Lambda}$ is a fixed point of the tile substitution corresponding to
(\ref{tilesubst}),
the {\it substitution tiling dynamical system} $(\mathbb{X}(\mathcal{T}_{\Lambda}),\mathbb{R})$
is uniquely ergodic (Theorem~4.1 of \cite{lee2003consequences}).

%%%
%%%
\begin{defn}
A {\it cut and project scheme} $(G,H,\tilde{\mathcal{L}})$ (or {\it CPS} for brevity)
consists of a locally compact abelian group $G$ which is a countable union of compact subsets,
a locally compact abelian group $H$ and a lattice $\tilde{\mathcal{L}}$ in $G\times H$, so that
the natural projections $\pi_{G}:G\times H\to G,\pi_{H}:G\times H\to H$ satisfy
\begin{itemize}
\item[(1)]$\pi_{G}\vert_{\tilde{\mathcal{L}}}$ is injective,
\item[(2)]$\pi_{H}(\tilde{\mathcal{L}})$ is dense in $H$.
\end{itemize}
A CPS $(G,H,\tilde{\mathcal{L}})$ is {\it symmetric} if $(H,G,\tilde{\mathcal{L}})$
is a CPS as well.
Setting $\mathcal{L}=\pi_{G}(\tilde{\mathcal{L}})$,
the {\it star-map} is defined by $(\cdot)^{\star}=\pi_{H}
\circ (\pi_{G}\vert_{\tilde{\mathcal{L}}})
^{-1}:\mathcal{L}\to H$.
\end{defn}

%%%
%%%
\begin{thm}[Proposition~6.3 of \cite{minervino2014geometry}]
$(\mathbb{R},K_{\sigma},\Phi(V\cdot\mathbb{Z}[\alpha^{-1}]))$
is a symmetric CPS with the star-map $\xi^{\star}=\Phi^{\prime}(\xi)$.
\end{thm}

The expansive matrix function system  $\Theta$ can be extended to the
{\it iterated function system}
$\Theta^{\star}=(\Theta_{ab})_{a,b\in\mathcal{A}}$ 
relative to the CPS $(\mathbb{R},K_{\sigma},\Phi(V\cdot\mathbb{Z}[\alpha^{-1}]))$
(\S 6.6 of \cite{sing2006pisot})
where we replace $g_{0}$ and $t_{p}$ in (\ref{theta_ab}) by 
$g_{0}(\xi^{\star})=\alpha^{\star} \xi^{\star}$ and $t_{p}(\xi^{\star})=
\xi^{\star}+\langle f(p),\vector{v}\rangle^{\star}$ respectively, and by extension
$\Theta_{ab}$ can be regarded as a set of mappings 
on $K_{\sigma}$
because $\Phi^{\prime}(V\cdot\mathbb{Z}[\alpha^{-1}])$ is dense in $K_{\sigma}$.
If $\mathcal{H}(K_{\sigma})$ denotes the space of non-empty compact subsets of $K_{\sigma}$,
then $\Theta^{\star}$ acts on $\mathcal{H}(K_{\sigma})^{n}$.
It is known for $\Theta^{\star}$
to admit an unique
attractor  (Corollary~6.63 of \cite{sing2006pisot})
: $\Theta^{\star}(\underline{\Omega})=\underline{\Omega}$.

%%%
%%%
\begin{defn}
Given a CPS $(G,H,\tilde{\mathcal{L}})$ and a subset $W\subset H$, define
\[
\Lambda(W)=\{ \xi\in\mathcal{L}:
\xi^{\star}\in W\}.
\]
If $W$ is non-empty compact and if $W=\overline{\mathrm{int}W}$,then
$\Lambda(W)$ is called a {\it model set}, and
$W$ is referred to as the ({\it acceptance}) {\it window} of the model set. 
A model set $\Lambda(W)$
is {\it regular} if the Haar measure of $\partial W$ is zero.
A subset $Q$ of a model set $\Lambda(W)$ is an {\it inter model set} if
$\Lambda(\mathrm{int}W)\subset Q\subset \Lambda(W)$.
\end{defn}

Consider the symmetric CPS $(K_{\sigma},\mathbb{R},\Phi(V\cdot\mathbb{Z}[\alpha^{-1}]))$.
Its star-map is defined by $(\cdot)^{\star}=\pi_{1}
\circ (\pi_{2}\vert_{\Phi(V\cdot\mathbb{Z}[\alpha^{-1}])})
^{-1}$.
To $(K_{\sigma},\mathbb{R},\Phi(V\cdot\mathbb{Z}[\alpha^{-1}]))$,
one can associate a regular inter model set 
\[
\Upsilon_{a}:=
\Lambda([0,\langle\vector{e}_{a},\vector{v}\rangle))=\{ X=\Phi^{\prime}(\xi)
\in \Phi^{\prime}( 
V\cdot\mathbb{Z}[\alpha^{-1}])
: X^{\star}=\xi\in [0,\langle\vector{e}_{a},\vector{v}\rangle)\}
\]
for $a\in\mathcal{A}$. Write
$\underline{\Upsilon}=(\Upsilon_{a})_{a\in\mathcal{A}}$.

%%%
%%%
\begin{thm}[(6.4) and Theorem~7.7 of \cite{minervino2014geometry}]\label{consistency}
Let $\sigma$ be an irreducible Pisot substitution. Then
\[
\underline{\Omega}=(\Omega_{a})_{a\in\mathcal{A}}=
(\mathcal{R}_{\sigma}(a))_{a\in\mathcal{A}}
\quad \mathrm{and}\quad \underline{\Upsilon}=
(\Upsilon_{a})_{a\in\mathcal{A}}
=(\Gamma_{a})_{a\in\mathcal{A}}.
\]
\end{thm}

%%%
%%%
\begin{lem}\label{associated_tiling}
$\underline{\Upsilon}+\underline{\Omega}=:
\bigcup_{a\in\mathcal{A}}\{ \Omega_{a}+\gamma: \gamma\in \Upsilon_{a}\}$
equals $\mathcal{T}$ and thus is a tiling.
\end{lem}
\begin{proof}
It is obvious from Theorem~\ref{trulytiling} and Theorem~\ref{consistency}.
\end{proof}

Define
\[
\mathbb{X}(\underline{\Lambda})=\overline{\{-t+\underline{\Lambda}:t\in\mathbb{R}\}}
\] 
(the orbit closure of a point $\underline{\Lambda}$ by $\mathbb{R}$ action with respect to
the {\it local metric}. See Definition~5.102 of \cite{sing2006pisot} or (2,2) 
of \cite{lee2003consequences}).
The  {\it point set dynamical system}
$(\mathbb{X}(\underline{\Lambda}),\mathbb{R})$ is known to be uniquely ergodic.

%%%
%%%
\begin{thm}[Theorem 6.116 of \cite{sing2006pisot}]\label{pointset}
Let $\sigma$ be an irreducible Pisot substitution.
The followings are equivalent.
\begin{itemize}
\item[(1)]$\underline{\Upsilon}+\underline{\Omega}$ is a tiling.
\item[(2)] The point set dynamical system $(\mathbb{X}(\underline{\Lambda}),\mathbb{R})$
has pure point spectrum.
\item[(3)]Each $\Lambda_{a}$ is an inter model set with $\Lambda(\mathrm{int}
\Omega_{a})\subset\Lambda_{a}\subset\Lambda(\Omega_{a})$.
\end{itemize}
\end{thm}

Since $(\mathbb{X}(\mathcal{T}_{\Lambda}),\mathbb{R})$ and $(\mathbb{X}(\underline{\Lambda})
,\mathbb{R})$ are topologically conjugate (Lemma~5.115 of \cite{sing2006pisot}
or Lemma~3.10 of \cite{lee2003consequences}), Lemma~\ref{associated_tiling}
and Theorem~\ref{pointset} imply the following.

\begin{thm}\label{substiling}
For an irreducible Pisot substitution,
the {\it substitution tiling dynamical system} $(\mathbb{X}(\mathcal{T}_{\Lambda}),\mathbb{R})$
has pure point spectrum.
\end{thm}

%%%
%%%
\begin{thm}[\S 4 of \cite{akiyama2015pisot} and Theorem~3.1 of \cite{clark2003size}]
\label{botheq}
For an irreducible Pisot substitution, 
the substitution dynamical system has pure point spectrum if and only if the 
substitution tiling dynamical system does so.
\end{thm}

Combining Theorem~\ref{substiling} with Theorem~\ref{botheq},
we conclude the following.
%%%
\begin{thm}For an irreducible Pisot substitution,
the substitution dynamical system $(\overline{\mathcal{O}_{\sigma}(u)},S,\nu)$
has  pure point spectrum (or pure discrete spectrum).
\end{thm}

\section{Strong Coincidence Conjecture}\label{strongcc}

In this section, we will show that tiling implies strong coincidence.
For unimodular cases, this has already been established in \cite{ito2006atomic}.

A \textit{patch} of $\Gamma$ is a finite subset of $\Gamma$.
By abuse of language, the subcollection of $\mathcal{T}$
associated with a patch $\Gamma_{0}$
\[
\{ \mathcal{R}_{\sigma}(b)+\gamma:(\gamma,b)\in\Gamma_{0}\}
\]
is also referred to as a \textit{patch}.
A \textit{translation} of the patch $\Gamma_{0}$ 
means 
\[
\{ (\gamma+t,b)\in\Gamma:(\gamma,b)\in\Gamma_{0}\}\quad\mathrm{for\ some}\
t\in K_{\sigma}.
\]

%%%
%%%
%%%
\begin{thm}[\cite{minervino2014geometry},\cite{sing2006pisot} and
\cite{ito2006atomic}]\label{quasi}
$\Gamma$ is quasi-periodic, i.e. for any patch of $\Gamma$
there exists $R>0$ so that
every ball of radius $R$ in $K_{\sigma}$ contains the first coordinates of
a translation of this patch.
\end{thm}

It is easy to see that
$(\pi_{2}\times id)\circ T^{-k}_{\rm ext}(\vector{0},a)\subset \Gamma$,
where $id$ means the identity map from $\mathcal{A}$ to itself

For $\bf{B}\subset K_{\sigma}$, define
\[
\Gamma_{\bf{B}}=\{ (\gamma, b)\in \Gamma: \gamma\in \bf{B}\}.
\]

%%%
%%%
%%%
\begin{lem}\label{ball}
For each $a\in\mathcal{A}$ and $k\geq 1$, 
there exists a ball $\bf{B}$ so that
$\Gamma_{\bf{B}}\subset (\pi_{2}\times id)
\circ  T^{-k}_{\rm ext}(\vector{0},a)$ and that
the radius of $\bf{B}$ tends to infinity
as $k\to\infty$.
\end{lem}
\begin{proof}
Since $\mathcal{R}_{\sigma}(a)$ is interior-dense, one can choose
a sequence of balls $B(Z_{k},R_{k})$ contained in $\alpha^{-k}\mathcal{R}_{\sigma}(a)$
so that the radius $R_{k}$ goes to infinity as $k\to\infty$.
For each $k\geq 1$, it follows by (\ref{iterate}) that
\[
\{ \mathcal{R}_{\sigma}(b)+\pi_{2}(Y^{*}): (Y^{*},b)\in T^{-k}_{\rm ext}(\vector{0},a) \}
\]
is a patch. By the set equation (\ref{seteqI}), this patch tiles
$\alpha^{-k}\mathcal{R}_{\sigma}(a)$ and hence $B(Z_{k},R_{k})$.
Since $\mathrm{int}\mathcal{R}_{\sigma}(a^{\prime})\cap
\mathrm{int}\mathcal{R}_{\sigma}(a)=\emptyset$ for $a^{\prime}\neq a$ (tiling),
it follows that
\[
(\mathcal{R}_{\sigma}(b)+\pi_{2}(Y^{*}))\cap B(Z_{k},R_{k})=\emptyset
\]
for any $(Y^{*},b)\not\in T^{-k}_{\rm ext}(\vector{0},a)$ with $(\pi_{2}(Y^{*}),b)\in\Gamma$.
This implies that $d_{K}(Z_{k},Z+\pi_{2}(Y^{*}))\geq R_{k}$ for $Z\in
\mathcal{R}_{\sigma}(b)$.
If $C_{1}>0$ is taken so that $\mathcal{R}_{\sigma}\subset B(0,C_{1})$,
then $d_{K}(0,Z)\leq C_{1}$, and the triangle inequality
\[
d_{K}(Z_{k},Z+\pi_{2}(Y^{*}))\leq d_{K}(Z_{k},Z+Z_{k})+d_{K}(Z+Z_{k},Z+\pi_{2}(Y^{*}))
\]
\[
=d_{K}(0,Z)+d_{K}(Z_{k},\pi_{2}(Y^{*}))
\]
suggests that
$\pi_{2}(Y^{*})\not\in B(Z_{k},R_{k}-C_{1})$.
Consequently
$\Gamma_{B(Z_{k},R_{k}-C_{1})}\subset (\pi_{2}\times id)
\circ  T^{-k}_{\rm ext}(0,a)$ and $R_{k}-C_{1}\rightarrow\infty$
as $k\rightarrow\infty$,.
\end{proof}

%%%
%%%
%%%
\begin{thm}Let $\sigma$ be an irreducible Pisot substitution.
Then $\sigma$ satisfies the strong coincidence condition.
\end{thm}
\begin{proof}
Let $B(0,R_{0})$ be a ball of radius $R_{0}>0$ at $0$.
Then $\Gamma_{B(0,R_{0})}$ is a patch since
$\Gamma$ is a multi-component Delone set. 

By Lemma~\ref{ball},
one can take  $k>0$ so large that
the radius of the ball $B$ exceeds $R$ in Theorem~\ref{quasi}
with $\Gamma_{B}\subset (\pi_{2}\times id)
\circ T^{-k}_{\rm ext}(\vector{0},1)$.

The quasi-periodicity of $\Gamma$ (Theorem~\ref{quasi})
yields that the first coordinates of a translation of $\Gamma_{B(0,R_{0})}$
are contained in $B$.
In particular, since $(0,i),(0,j)\in\Gamma_{B(0,R_{0})}$,
there exists a translation $t\in K_{\sigma}$ so that $(t,i), (t,j)\in
\Gamma_{B}\subset(\pi_{2}\times id)
\circ T^{-k}_{\rm ext}(\vector{0},1)$.
This implies that there exist prefixes $P^{(i)}$ of $\sigma^{k}(i)$ and $P^{(j)}$ of $\sigma^{k}(j)$ 
so that
$\sigma^{k}(i)=P^{(i)}1s$ and $\sigma^{k}(j)=P^{(j)}1s^{\prime}$ where $s$ and $s^{\prime}$ 
are suffixes.
Besides
\begin{equation}\label{t}
\begin{split}
t&=\pi_{2}(\alpha^{-k}
\Phi(\langle f(P^{(i)}),\vector{v}\rangle))=
\Phi^{\prime}(\langle M_{\sigma}^{-k}f(P^{(i)}),\vector{v}\rangle) \\
&=
\pi_{2}(\alpha^{-k}
\Phi(\langle f(P^{(j}),\vector{v}\rangle))=
\Phi^{\prime}(\langle M_{\sigma}^{-k}f(P^{(j)}),\vector{v}\rangle).
\end{split}
\end{equation}
Corollary~\ref{latticeII} shows that
(\ref{t}) implies that $f(P^{(i)})=f(P^{(j)})$.  
This completes the proof.
\end{proof}

%%%\acknowledgemen
%%%Last but not least, this paper is dedicated to the memory of Professor Shunji Ito, 1948\,--\,2021 
%%%

\bibliographystyle{amsplain}
\bibliography{myref}

\end{document}